\documentclass[11pt]{article}
\usepackage{amsthm,amsfonts,amsmath,amscd,amssymb,times,epsfig,verbatim}
\usepackage[all]{xy}
\usepackage{enumerate,array}
\usepackage{tikz-cd} 
\usepackage{color}

\newtheorem{thm}{Theorem}
\newtheorem{prop}{Proposition}
\newtheorem{lem}{Lemma}
\newtheorem{cor}{Corollary}

\theoremstyle{remark}
\newtheorem{rem}{Remark}
\theoremstyle{definition}
\theoremstyle{axiom}

\newtheorem{ex}{Example}

\newcommand{\C}{\mathbb{C}}

\newcommand{\R}{\mathbb{ R}}

\newcommand{\rk}{\operatorname{rk}}

\newcommand{\Z}{\mathbb{ Z}}

\title{Smooth manifolds in  $G_{n,2}$ and $\C P^{N}$ defined by symplectic reductions of  $T^n$-action}

\author{Victor M.~Buchstaber and Svjetlana Terzi\'c}

\begin{document}

\maketitle

\begin{abstract}
We study   symplectic reductions arising from the  canonical action of the maximal compact torus $T^n$  on the  complex Grassmann manifold $G_{n,2}$ as well as those arising from the   $T^n$-action on the complex projective space $\C P^{N}$, $N=\binom{n}{2}-1$ for which the   Pl\"ucker embedding  $G_{n,2}\to \C P^{N}$ is   $T^n$-equivariant.  We investigate   the topology of regular level sets of the  moment maps and the corresponding symplectic reductions.  

 For $n=4$  we show  that the regular level sets of the moment maps  do not depend on a regular value.  We prove this set to be    homeomorphic to $S^3\times T^2$ in the case $G_{4,2}$, while   in the  case $\C P^5$  it is homeomorphic to  $ S^5\times T^2$.

We relate our constructions to moduli spaces of weighted pointed stable genus zero  curves. The Deligne-Mumford compactification $\overline{\mathcal{M}}_{0, n}$ is proved to arise  as a  symplectic reduction of $G_{n,2}$ by the canonical $T^n$-action  in precisely the cases  $n=4,5$. In the case $n\geq 5$, there is well known  the Losev-Manin compactification different from Deligne-Mumford and it  appears to be this symplectic reduction only  in the case $n=5$. In this way we show that for $n=5$,  a symplectic reduction depends on a regular value of the moment map.

\end{abstract}

\section{Introduction}
The complex  Grassmann manifolds $G_{n,2}$  of two-dimensional complex subspaces in $\C ^{n}$ are of  specific mathematical interest as the  canonical actions of an algebraic torus $(\C ^{\ast})^n$ as well as the   compact torus $T^n$ on these manifolds are  closely related to many important problems.  One of them is a problem on   moduli compactifications of the moduli space of $n$ ordered   distinct points on $\C P^1$,  in particular those compactifications given by the weighted pointed  stable genus zero curves,~\cite{H},~\cite{BTW}. In addition,  in~\cite{HK1},~\cite{HK2},  the moduli spaces $\text{Pol}(m)$  of $m$-gons in $\R^3$ are proved to be homeomorphic to $G_{n,2}/T^n$, while the moduli spaces of $\text{Pol}(a_1, \ldots , a_m)$ of $m$-gons with prescribed $i$-th side length $a_i$, are proved to be homeomorphic to a symplectic quotient of $G_{n,2}$ by the canonical $T^n$-action.

Many structures on $G_{n,2}$   can be assigned to these  torus actions. There is the standard moment map $\mu : G_{n,2}\to \Delta _{n,2}$ for the hypersimplex $\Delta _{n,2}\subset \R^n$. A regular value, in the classical sense,  of the map $\mu$ is known~\cite{BT3} to be a point $x\in \Delta _{n,2}$ such that the stabilizer of any  $y\in \mu ^{-1}(x)$, for the  $T^n$-action on $G_{n, 2}$, is the diagonal circle $S^1$. The level set $\mu ^{-1}(x)\subset G_{n,2}$ is a smooth submanifold and  the orbit space $\mu ^{-1}(x)/T^n$ is a symplectic manifold, known as a symplectic reduction for the given $T^n$-action.  In addition, the polytopes assigned to the strata on $G_{n,2}$ defined by the $(\C^{\ast})^n$-action, give the chamber decomposition of $\Delta _{n,2}$. It is proved that  for any chamber $C_{\omega}$ of maximal dimension $(n-1)$ all  level sets $\mu ^{-1}(x)$, $x\in C_{\omega}$ are diffeomorphic leading to the smooth manifold  $F_{\omega} = \mu ^{-1}(x)/T^n \subset G_{n,2}/T^n$, ~\cite{GMP},~\cite{BT3}. In addition for any chamber $C_{\omega}$ all  level sets $\mu ^{-1}(x)$, $x\in C_{\omega}$ are homeomorphic implying that $G_{n,2}/T^n =\cup _{\omega}C_{\omega}\times F_{\omega}$. Starting from this decomposition  we constructed in~\cite{BT3} the model $(U_n, p_n)$ for the orbit space $G_{n,2}/T^n$, where $U_n = \Delta_{n,2}\times \mathcal{F}_{n}$ for the  smooth compact manifold $\mathcal{F}_{n}$ and $p_n :U_n\to G_{n,2}$ is the projection which is a diffeomorphism between dense open sets in $U_n$ and $G_{n,2}$. The  smooth compact manifold $\mathcal{F}_{n}$ we call universal space of parameter and we proved in~\cite{BTCH} that it coincides with the Chow quotient $G_{n,2}\!/\!/(\C ^{\ast})^{n}$ defined by Kapranov in~\cite{KAP}, that is with the Deligne-Mumford compactification $\overline{\mathcal{M}}_{0, n}$ of the moduli space of $n$ ordered distinct points in $\C P^1$.

In this paper we study  the problem of explicit topological description of smooth manifolds $\mu ^{-1}({\bf x})$   of  a regular values ${\bf x}\in \Delta _{n,2}$ for the Grassmannians $G_{n,2}$.  The Pl\"ucker coordinates give an embedding of $G_{n,2}$ into the  complex projective space $\C P^{N}$, $N =\binom{n}{2}-1$, which is equivariant for the canonical $T^n$-action on $G_{n,2}$ and $T^n$-action on $\C P^{N}$ given by the composition of the second exterior power representation $T^n\to T^{N}$ and the standard $T^N$-action on $\C P^{N}$.   Preceding the announced problem, we study $\C P^{N}$ with the prescribed $T^n$-action and the corresponding moment map $\tilde{\mu} : \C P^{N}\to \Delta _{n,2}$. The regular values  of  $\tilde{\mu}$ are    the  regular values of  the moment map $\mu$ as well, but we show  that the vice versa does not hold in general.

Let $X_n\subset \Delta _{n,2}$ be the set of points consisting  of regular values for both $\mu$ and $\tilde{\mu}$. For any point ${\bf x} \in X_{n}$ we have a pair of smooth manifolds $M_{1}\subset G_{n,2}$ and $M_{2}\subset \C P^{N}$ with the free actions  of the torus $T^{n-1} = T^{n}/T^1$, where $T^1=\text{diag}(T^n)$, 
and  $M_{1} = \mu ^{-1}({\bf x})$, $\dim M_1 = 3n-7$  and $M_{2} = \tilde{\mu}^{-1}({\bf x})$, $\dim M_{2} =n^{2}-2n-1$. The Pl\"ucker embedding induces an  equivariant embedding $M_{1}\subset M_{2}$   as well.

The smooth manifolds $M_{1}/T^n$ and $M_{2}/T^n$ are symplectic manifolds and they are known as symplectic reduction for $T^n$-actions on $G_{n,2}$ and $\C P^{N}$ respectively.

 The first problem we address  is: describe the structure of $T^n$-pair of symplectic manifolds $M_{1}\subset M_{2}$.

 For many years, it has been  widely known the problem to  describe  a structure of the set of modular compactification for the moduli space $\mathcal{M}_{0, n}$ 
of ordered distinct  $n$ points on $\C P^1$.  Among such compactifications a special attention is given to the Deligne- Mumford compactification $\overline{\mathcal{M}}_{0, n}$ (1969) and the Losev-Manin compactification $\bar{L}_{0, n}$ (2000).  The Losev-Manin spaces are   smooth toric manifolds as well~\cite{BTW},~\cite{M}.   Based on~\cite{H}, a category was   introduced in our paper~\cite{BTW}  which we call  Hassett category, whose objects are moduli compactifications of $\mathcal{M}_{0, n}$. These objects are given by  moduli spaces of stable pointed genus zero curves $\overline{\mathcal {M}}_{0, \mathcal{A}}$,  where   weight vectors $\mathcal{A}\in \R^n$ are close enough to vectors in a chamber of maximal dimension in $\Delta _{n,2}$.  The initial object in  this  category is the Deligne-Mumford compactification.  Related to this,   it manifests  a specific of the Grassmann manifolds $G_{n,2}$, that the space of parameters of a chamber of maximal dimension in $\Delta _{n,2}$  is isomorphic to a moduli space 
$\overline{\mathcal{M}}_{0, \mathcal{A}}$, ~\cite{BTW},~\cite{H}. Using this,  we show that a  symplectic reduction $M_{1}/T^n$ belongs to the Hassett category as well.

The second problem we address  is : describe symplectic reductions $M_{1}/T^n$,  which can be realized as Deligne-Mumford compactifications  or Losev-Manin compactifications. 

In this paper the first problem is solved for $n=4$. In this case the regular values of the corresponding  moment maps for $G_{4,2}$ and $\C P^5$ coincide and in both cases we show that the corresponding smooth manifolds  do not depend on a regular value of $\mu$, that is $\tilde{\mu}$.   We prove that $\tilde{\mu} ^{-1}({\bf x}) \cong S^5\times T^2$ for any regular value  ${\bf x}\in \Delta _{4,2}$. Starting from  this we prove that $\mu ^{-1}({\bf x}) \cong S^3\times T^2$ for any regular value  ${\bf x}\in \Delta _{4,2}$.

In the framework of the second problem,  the symplectic reductions $M_1/T^n$ are described as objects of Hassett category, moreover, it is described  the largest one. We prove that the Deligne-Mumford compactification  $\overline{\mathcal{M}}_{0,n}$ and Losev-Manin compactification $\bar{L}_{0,n}$ can be realized as a symplectic reduction $M_{1}/T^n$ if and only if $n=4$ and $5$. In particular,   for $n=5$ we obtain two such  non-isomorphic    symplectic reductions given by the spaces  $\mathcal{M}_{0,5}$ and $\bar{L}_{0,5}$, which shows that in general, a symplectic reduction $M_1/T^n$, $n\geq 5$ does depend on a regular value of the moment map.
   
The organization of the paper is as follows. In Section 3 we provide short background on the canonical $T^n$-action on the Grassmannians $G_{n,2}$ as well as on the $T^n$-action on the complex projective spaces $\C P^{N}$, $N=\binom{n}{2}-1$  for which  the Pl\"ucker embedding  $G_{n,2}\to \C P^N$  is equivariant.  In Section 4 the regular level set of the moment map for the $T^5$-action on $\C P^5$ is described in detail.   It is proved that this set  does not depend on a regular value and that it is homeomoprhic to $S^5\times T^2$. Section 5 investigates    the canonical $T^4$-action  on the  regular level set $M_{Q}^{5}\subset G_{4,2}$ for the standard moment nap. It is described the image of  $M_{Q}^{5}$  by the moment map arising from $T^4$-action on $\C P^5$ and it  is proved that the free $T^3=T^4/\text{diag}(T^4)$ -action on $M_{Q}^5\subset G_{4,2}\subset  \C P^{5}$ cannot be realized by a representation $T^3\to T^6$.  In Section 6, an explicit topological description of $M_{Q}^{5}$,  is obtained, both as a $T^3$-orbit of a two-dimensional surface $M^2$  and as a $T^2$-orbit of three-dimensional surface $M^3$ . The structure  of the principal fiber  bundle $M_{Q}^{5}\to \C P^1$, which arises from the free $T^3$-action on $M_{Q}^5$ is studied in detail in Section 7.  In Section 8 we describe the relations between  $\C P^1$,  $M^2$  and $M^3$ which lead to the proof that  $M_{Q}^5$ is homeomorphic to $S^3\times T^2$. We show in Section 9  that the manifold $M_{Q}^{5}\subset \R^8$  can be realized as a complete intersection. In Section 10, in the context of a symplectic  reduction,  we relate our previous constructions  on description of the     orbit space $G_{n,2}/T^n$ to moduli spaces of weighed pointed stable genus zero  curves.

\section{Acknowledgment} The authors are grateful to Taras Panov for stimulating us to work on this problem, as well as for many useful discussions and references. We thank the anonymous referee for carefully reading the paper and for valuable comments, which significantly improved the exposition of the paper.

\section{Short background}

We recall some necessary  results  and related constructions from ~\cite{BT1},~\cite{BT2},~\cite{BT3}, \cite{GMP}.
The Grassmann manifolds $G_{n,2}$ consists of $2$-dimensional complex subspaces in the complex vector space $\C ^{n}$. The coordinate wise $(\C ^{\ast})^n$-action and the induced $T^n$-action on $\C ^n$ produce the actions of these tori on $G_{n,2}$. The moment map for this $T^n$-action is defined in the standard way~\cite{K}.
Consider  the Pl\"ucker embedding  $p: G_{n,2}\to \C P^{N}$, $L\to (P^{I}(L))$, $I \in \binom{[n]}{2}$ and the second exterior power representation $T^n\to T^N$. The weight vectors for this representation are $\Lambda _{I} \in \Z^{n}$, $I\in \binom{[n]}{2}$ such that $\Lambda _{I}(j) =1$ for $j\in I$ and $\Lambda _{I}(j)=0$ for $j\not\in I$. The moment map $\mu : G_{n,2}\to \R^n $ is  given by
\[
\mu (L) = \frac{1}{\sum\limits_{I \in \binom{[n]}{2}}|P^{I}(L)|^2}\sum\limits_{I \in \binom{[n]}{2}}|P^{I}(L)|^{2}\Lambda _{I}.
\]
Its image is a hypersimplex $\Delta _{n,2}$, which is the convex hull of  the vertices $\Lambda _{I}$, $I \in  \binom{[n]}{2}$. In more detail, 
\[\Delta _{n,2} = \{{\bf x}\in \R^n |  0\leq x_i\leq 1, \; i=1,\ldots, n,  \;\; x_1+\ldots +x_n=2\}.\]

The moment map $\tilde{\mu}  : \C P^{N} \to \Delta_{n,2}$ defined by the prescribed $T^n$-action is, in the standard way~\cite{K}, given   by
\begin{equation}\label{momentmap}
\tilde{\mu}({\bf z}) = \frac{1}{\sum\limits_{i=0}^{N}|z_i|^2}\sum\limits_{i=0}^{N}|z_i|^2\Lambda_{I},
\end{equation}
where ${\bf z} = (z_0:\ldots :z_{N})$, and $\{0, \ldots , N\}$ and   $\{I, I\in \binom{[n]}{2}\}$   are related by: to the number $i$,  it corresponds the pair $kl$, which is at the $i$-place in lexicographical order for $\{I, I\in \binom{[n]}{2}\}$, that is   $(0,1,\ldots , N) = (12,13,\ldots , (n-1)n)$.

Note that it holds:

\begin{equation}\label{mutilde}
\mu = \tilde{\mu}\circ p.
\end{equation}
Moreover,  the moment map $\mu$ decomposes as follows: consider the linear map $A : \Delta ^{N}\to \Delta _{n,2}$ defined  using lexicographical order on the indices $I$ of the vertices $\Lambda _{I}$, that is
\begin{equation}\label{A}
A(0) =\Lambda _{12}, \; A(1) =\Lambda _{13}, \ldots , A(N) = \Lambda_{(n-1)n},
\end{equation}
and the standard moment map $\hat{\mu} : \C P^{N}\to \Delta ^{N}$. 
Then 
\begin{equation}\label{composition}
\mu = A \circ  \hat{\mu} \circ p.
\end{equation}

\begin{rem}
It follows from~\eqref{mutilde} that a regular value of the  moment map $\tilde{\mu}$ is a regular value of the moment map $\mu$. The vice versa  does not hold as we will demonstrate below.
\end{rem} 

\begin{rem}\label{regularvalues}
It  is proved in~\cite{BT2} that a point ${\bf x}\in \Delta _{n,2}$  is a regular value of the moment map $\tilde{\mu}  : \C P^{N}\to \Delta _{n,2}$   if and only if
the stabilizer of any point from $\tilde{\mu}^{-1}({\bf x})\subset \C P^{N}$ for the given $T^n$-action is $S ^{1} = \text{diag}(T^n)$. It is proved  that the same holds for a regular value of the moment map $\mu : G_{n,2}\to \Delta _{n,2}$.
\end{rem}

The $(\C ^{\ast})^{n}$-action on $\C P^{N}$ defines its stratification, that is the decomposition of $\C P^{N}$ into the strata. 

One of the equivalent definitions of a stratum would be:  for any  $\sigma \subset \{0, \ldots ,N\}$,    the   set 
$W_{\sigma} = \{{\bf z}\in \C P^{N} | z_i\neq 0 \; \text{iff}\; i\in \sigma\}$ is a stratum.   All points of  a stratum $W_{\sigma}$ have the same stabilizer $T_{\sigma}\subset T^n$, that is, the torus  $T^{\sigma} = T^{n}/T_{\sigma}$ acts freely on $W_{\sigma}$. The image of a stratum $W_{\sigma}$ by the moment map~\eqref{momentmap} is the interior of the  polytope $P_{\sigma}$  which is the convex hull of the vertices $\{\Lambda _{kl}\}$ standing by $z_i$ in~\eqref{momentmap} for $i\in \sigma$. Such polytopes will be called admissible. In this case an admissible polytope is any convex  polytope whose  set of vertices is a  subset of the set of vertices for $\Delta _{n,2}$.

 In particular, it is proved in~\cite{BT2} that $\dim T^{\sigma}=\dim P_{\sigma}$. 

The admissible polytopes $\{P_{\sigma}\}$  define the chamber decomposition $\{C _{\omega}\}$ of $\Delta _{n,2}$, by all possible non-empty intersections of $P_{\sigma}$. We deduce from Remark~\ref{regularvalues} that the regular values   of the moment map~\eqref{momentmap} are given by the points from the chambers  $C_{\omega}$ of maximal dimension $(n-1)$. In other words,  $M^{2N-n+1}_{{\bf x }} = \tilde{\mu} ^{-1}({\bf x }) \subset \C P^{N}$ is a smooth $(2N-n+1)$-dimensional manifold for any point ${\bf x}\in C_{\omega}$, $\dim C_{\omega} =n-1$ and only for such points.

Analogously, for a Grassmann manifold $G_{n,2}$ a stratum  is defined as an non-empty set $W_{\sigma} = \{L \in G_{n,2} | P^{ij}(L)\neq 0 \; \text{iff}\; \{i, j\}\in \sigma\}$, where $\sigma \subset \{\binom{[n]}{2}\}$. It holds as well that  all  points of  a stratum $W_{\sigma}$ have the same stabilizer $T_{\sigma}\subset T^n$, that is,  the torus  $T^{\sigma} = T^{n}/T_{\sigma}$ acts freely on $W_{\sigma}$. The image of a strata $W_{\sigma}$  is the  interior of the polytope $P_{\sigma}$ which is the  convex hull of vertices $\Lambda _{ij}$, $\{i, j\}\in \sigma$. The polytopes which can be obtained in this way are called admissible. All possible non-empty  intersections of admissible polytopes produce  the chamber decomposition for $\Delta _{n,2}$. 
Then  $M^{3n-7}_{\bf x} = \mu ^{-1}({\bf x})$ is a smooth manifold for any ${\bf x}\in \stackrel{\circ}{\Delta} _{n,2}$ which belongs to a chamber of maximal dimension $(n-1)$. It holds 
\[
M_{\bf x}^{3n-7} = M_{\bf x}^{2N-n-1}\cap p(G_{n,2}),\] for the Pl\"ucker embedding $p : G_{n,2}\to \C P^{N}$.

\begin{rem}\label{arrangement}
It is proved in~\cite{BT1} that  the   chamber decomposition of   $\Delta_{n,2}$ defined by the admissible polytopes for    the canonical $T^n$-action on $G_{n,2}$ coincides with the  lattice of the arrangement \[\mathcal{A} = \{x_{i_1}+ \ldots +x_{i_s}=1,  1\leq i_1<\ldots <i_s\leq n, \; 2\leq s\leq [\frac{n}{2}]\}\] intersected with the hypersimplex $\Delta _{n,2}$.
\end{rem}

\begin{rem}
The considered $T^n$-actions on $G_{n,2}$ and $\C P^{N}$ produce   two families of admissible polytopes and  two  chamber decompositions  for $\Delta _{n,2}$ which  do not coincide for $n\geq 5$.
\end{rem}

\begin{ex} 
Let $n=5$, that is $N=9$ and consider a point ${\bf z} = (z_0:z_1:0:0:z_4:0:0:0:0:z_{9})$, where $z_0, z_1, z_4, z_9\neq 0$, The admissible polytope $P_{\bf z}$ of the stratum in $\C P^{9}$  which contains ${\bf z}$  is a convex hull of the vertices $\Lambda _{12}, \Lambda _{13}, \Lambda _{23}, \Lambda_{45}$ and its dimension is $3$. On the other hand,  it follows from Remark~\ref{arrangement} that any $3$-dimensional admissible polytope for $G_{5,2}$ belongs to a hyperplane $x_i+x_j=1$ for some $1\leq i<j\leq 5$. Obviously none of these hyperplanes contains $P_{\bf z}$,  so $P_{\bf z}$ is not an admissible polytope for $G_{5,2}$.  Thus,  the polytope $P_{\bf z}$ intersects a  chamber in $\Delta _{5,2}$ of maximal dimension $4$, which  comes  from   $G_{5,2}$  of maximal dimension  $4$. In particular, the chamber decompositions of $\Delta _{5,2}$ coming from $G_{5,2}$ and $\C P^{9}$ are different. According to Remark~\ref{regularvalues}, this also shows that a  point from $P_{\bf z}$ is a regular value for the moment map $\mu$, but it is not a regular value for the moment map $\tilde{\mu}$.
\end{ex}   

\section{$T^4$-action on $\C P^5$}

Consider the representation of $T^4$ in $T^6$ given by the second exterior power, that is  $(t_1, t_2, t_3, t_4) \to (t_1t_2, t_1t_3, t_1t_4, t_2t_3, t_2t_4, t_3t_4)$. This representation composed with the canonical $T^6$-action on $\C P^5$ defines $T^4$-action on $\C P^5$.  Let  ${\bf z} = (z_0 : z_1 : z_2 : z_3 : z_4 :z_5)$ be the  homogeneous coordinates on $\C P^5$. The moment map $\tilde{\mu }: \C P^5 \to \R^4 $, for this $T^4$-action   is given   by
\begin{equation}\label{momentmap4}
\tilde{\mu}  ({\bf z}) =\frac{1}{\|{\bf z}\|}( |z_0|^2 \Lambda _{12} + |z_1|^2\Lambda _{13} + |z_2|^2 \Lambda _{14} + |z_3|^2 \Lambda _{23}+ |z_4|^2\Lambda _{24} + |z_5|^2 \Lambda _{34}).
\end{equation}
  The image of $\tilde{\mu}$ is the  hypersimplex  $\Delta _{4,2}$ , which  is an octahedron. Its dimension is  $3$.

 For a chamber $C_{\omega}\subset \Delta _{4,2}$ of maximal dimension $3$, any point ${\bf x}\in C_{\omega}$ is a regular value of the moment map~\eqref{momentmap4}, that is $M^{7}_{{\bf x }} = \tilde{\mu} ^{-1}({\bf x }) \subset \C P^{5}$ is a smooth $7$-dimensional manifold.

This chamber decomposition is given as well by the lattice of the arrangement $\mathcal{A} = \{x_1+x_i =1, i=2,3,4\}$ intersected with the octahedron $\Delta _{4,2}$. There are $8$ chambers of dimension $3$ and they are given by the intersection of the halfspaces  $x_1+x_2 <(>)1, x_1+x_3< (>)1, x_1+x_4<(>)1$.  The standard action of the symmetric group $S_4$ on $\R^4$  induces the $S_4$-action on  $\Delta _{4,2}$.

\begin{lem}
The $S^4$-action on the set of  $3$-dimensional chambers in $\Delta _{4,2}$ consists of two orbits whose representatives are the chamber  $C^{-}$ given by the intersection of the halfspaces  $x_1+x_i<1$, $ i=2,3,4$ and the chamber $C^{+}$ given by the intersection of the halfspaces $x_1+x_i>1$, $i=2,3,4$.
\end{lem}

\begin{proof}
A permutation from $S_4$ maps  the chamber $\cap \{x_1+x_i<1, i=2,3, 4\}$ to a chamber $\cap \{x_j+x_k<1,  1\leq k\leq 4, k\neq j \}$  for some $1\leq j\leq 4$.  Therefore,  the orbit of this chamber consists of four chambers which are different  from the chambers $x_1+x_i>1$, $i=2, 3,4$. For the second one chamber  the same argument holds.
\end{proof}

\begin{lem}
The manifolds  $M^{7}_{{\bf x }}$ and $M^{7}_{{\bf y }}$ are diffeomorphic for any ${\bf x}, {\bf y}$ which belong to chambers of maximal dimension of the same $S_4$-orbit.
\end{lem}
\begin{proof}
An element for $S_4$ permutes the vertices of $\Delta _{4,2}$ and, 
 following the assignment $z_i\leftrightarrow \Lambda _{kl}$,  it   can be assigned  to any such permutation  the permutation in  $\C ^6$. In this way,  we see that if two chambers   of maximal dimension belong to the same $S_4$-orbit, then  there exists permutation from $S_6$ which maps the set of strata whose admissible polytopes form one chamber into the set of strata whose admissible polytopes form the other chamber. In particular, $M^{7}_{{\bf x }}$ and $M^{7}_{{\bf y }}$ are diffeomorphic for any ${\bf x}, {\bf y}$ which belong to chambers of maximal dimension of the same $S_4$-orbit.
\end{proof}

\subsection{A regular level set  of the moment map}
We now consider the chamber $x_1+x_i<1$, $i=2,3,4$. The point $Q= (\frac{1}{3}, \frac{5}{9}, \frac{5}{9}, \frac{5}{9})$  belongs to this chamber.  Denote by $M_{Q}^{7}$ the level set at $Q$ of  the moment map $\tilde{\mu}$.  According to~\cite{BT2},  the stabilizer of any point  $L\in M_{Q}^{7}$ for the given $T^4$-action is $S^1=\text{diag}(T^4)$, since $Q$ does not belong to the hyperplanes in $\R^4$ given by $x_i+x_j=1$, so    $M_{Q}^{7}$  is  a smooth manifold.

If we put  $Z= \sum\limits _{i=0}^{5}|z_i|^2$ , we see that  $M_{Q}^{7}\subset \C P^5$ is given by the following system:
\begin{equation}\label{definingequations}
3(|z_0|^2 + |z_1|^2 + |z_2|^2) = Z, \;\;\; 9(|z_0|^2 + |z_3|^2 + |z_4|^2) =5Z,
\end{equation}
\[
9(|z_1|^2 + |z_3|^2 + |z_5|^2) =5Z, \;\; 9(|z_2|^2 + |z_4|^2 + |z_5|^2) =5Z.
\]
We deduce that
\[
|z_0|^2 + |z_4|^2 = |z_1|^2 + |z_5|^2, \;\; 	|z_0|^2 + |z_3|^2 = |z_2|^2 + |z_5|^2,
\]
\[ |z_1|^2 + |z_3|^2 = |z_2|^2 + |z_4|^2.
\]
It follows
\[
|z_0|^2  + |z_3|^2 -|z_2|^2 = |z_0|^2 + |z_4|^2 - |z_1|^2 ,\]
which implies 
\begin{equation}\label{z41}
|z_4|^2 =|z_1|^2 -|z_2|^2 + |z_3|^2
\end{equation}
 and
\begin{equation}\label{z51}
|z_5|^2 = |z_0|^2-|z_2|^2 +|z_3|^2.
\end{equation}
Substituing this in the first equation of ~\eqref{definingequations} we obtain
\begin{equation}\label{z3}
|z_3|^2 = \frac{1}{3}(|z_0|^2 + |z_1|^2 + 4 |z_2|^2).
\end{equation}

Now, by straightforward substitution, from ~\eqref{z41} and~\eqref{z51} we deduce
\begin{equation}\label{z4z5}
|z_4|^2 = \frac{1}{3}(|z_0|^2 + 4|z_1|^2 +  |z_2|^2), \;\;\; |z_5|^2 = \frac{1}{3}(4|z_0|^2 + |z_1|^2 +  |z_2|^2).
\end{equation}

\begin{rem}
Note that for ${\bf z}\in M_{Q}^{7}$,  it holds $z_3, z_4, z_5\neq 0$.
\end{rem}

\begin{lem}\label{triangle}
The image $\hat{\mu}(M_{Q}^{7})$ by the standard moment map $\hat{\mu} : \C P^5\to \Delta ^{5}$ is the  triangle $P$  given by
\begin{equation}
 x_2  = \frac{1}{3} - x_0-x_1, \quad  x_3 = \frac{4}{9}-x_0-x_1, 
\end{equation}
 \[
x_4= \frac{1}{9}+x_1,\quad   x_5  = \frac{1}{9}+x_0\]

and the condition  $ x_0+x_1 \leq \frac{1}{3}$.

\end{lem}

\begin{proof}
Let ${\bf z} \in \hat{\mu} ^{-1}(Q)  $ and $x_{i} =\frac{|z_i|^2}{\|{\bf z}\|^{2}}$. Then $\hat{\mu}({\bf z})  = (x_0, \ldots , x_5)$.   It follows from~\eqref{z3},~\eqref{z4z5}  that 
\[
x_3 = \frac{1}{3}(x_0+x_1+4x_2),\;  x_4 = \frac{1}{3}(x_0+4x_1+x_2),\; x_5 = \frac{1}{3}(4x_0+x_1+x_2).
\]
 The equation $\sum\limits _{i=0}^{5}x_i=1$ implies that  
\[
3(x_0+x_1+x_2)=1,  \; \text{that is}\; x_{2} = \frac{1}{3}-x_0-x_1.
\]
Thus, $\hat{\mu} (M_{Q})$ is given by
\[
  x_2  = \frac{1}{3} - x_0-x_1, \; \; x_3 = \frac{4}{9}-x_0-x_1,
\]
\[ \ x_4= \frac{1}{9}+x_1,\;  \;  x_5  = \frac{1}{9}+x_0, 
\]
where $x_0, x_1\geq 0$ and $x_0+x_1 \leq \frac{1}{3}$.
\end{proof}

\begin{cor}\label{edges}
The edges of $P$ are:
\[
I_{0}:   (0, x_1, \frac{1}{3} - x_1, \frac{4}{9} - x_1, \frac{1}{9} -x_1, \frac{1}{9}), \; 0\leq x_1\leq \frac{1}{3}.\,
\]
\[
I_{1} : (x_0, 0, \frac{1}{3}-x_0, \frac{4}{9}-x_0, \frac{1}{9}, \frac{1}{9}+x_0), \; 0\leq x_1\leq \frac{1}{3},
\]
\[
I_{2} : (x_0, \frac{1}{3}-x_0, 0, \frac{1}{9}, \frac{4}{9} - x_0, \frac{1}{9}+x_0), \; 0\leq x_0\leq \frac{1}{3}.
\]
The vertices for $P$ are
\[
X_{01}  = (0, 0, \frac{1}{3}, \frac{4}{9}, \frac{1}{9},  \frac{1}{9}), \; X_{02} = (0, \frac{1}{3}, 0, \frac{1}{9}, \frac{4}{9}, \frac{1}{9}), \; 
X_{12} =  (\frac{1}{3}, 0, 0, \frac{1}{9}, \frac{1}{9}, \frac{4}{9}).
\]
\end{cor}

We can approach this in slightly different way. Let $v_i$ be the standard basis in $\R ^6$ and consider the map  $A : \mathbb{R}^6\to \mathbb{R}^4$, compare to~\eqref{A}, given by
\[
  Av_0 = \Lambda_{12}; \; Av_1 = \Lambda_{13}; \; Av_2 = \Lambda_{14}; \; Av_3 = \Lambda_{23}; \; Av_4 = \Lambda_{24}; \; Av_5 = \Lambda_{34}.
\]
In the considered  base the map $A$ is given by the matrix
\begin{equation}\label{1}
  A =
  \begin{pmatrix}
    1 & 1 & 1 & 0 & 0 & 0 \\
    1 & 0 & 0 & 1 & 1 & 0 \\
    0 & 1 & 0 & 1 & 0 & 1 \\
    0 & 0 & 1 & 0 & 1 & 1
  \end{pmatrix}.
\end{equation}
We need to solve the system
\begin{equation}\label{2}
  AX=Y,\quad X\in \mathbb{R}^6,\quad Y\in \mathbb{R}^4,
\end{equation}
where $Y$ is a column vector  $(\frac{1}{3}, \frac{5}{9}, \frac{5}{9}, \frac{5}{9})$. By the standard methods we obtain that the system~\eqref{2} is equivalent to the system 
\begin{equation}\label{3}
\begin{pmatrix}
  1 & 1 & 1 & 0 \\
  0 & -1 & -1 & 1 \\
  0 & 0 & -1 & 2 \\
  0 & 0 & 0 & 2
\end{pmatrix}
\begin{pmatrix}
  x_0 \\
  x_1 \\
  x_2 \\
  x_3
\end{pmatrix} = 
\begin{pmatrix}
  y^*_1 \\
  y^*_2 \\
  y^*_3 \\
  y^*_4
\end{pmatrix},
\end{equation}
where $y^*_1 = \frac{1}{3}, \quad y^*_2 = \frac{2}{9}-x_4, \quad y^*_3 = \frac{7}{9}-x_4-x_5, \quad y^*_4 = 2(\frac{2}{3}-x_4-x_5)$. We finally get
\begin{equation}\label{4}
x_0 = -\frac{1}{9}+x_5,\quad x_1 = -\frac{1}{9}+x_4,
\end{equation}
 \[
x_2 = \frac{5}{9}-x_4-x_5,\quad x_3 = \frac{2}{3}-x_4-x_5.
\]
By its construction the map $A$ gives the map $\Delta ^{5} \to \Delta_{4,2}$. Let $L\subset \R^6$ be an affine space which is the solution of~\eqref{2}. 

  \begin{lem}
The intersection  $P = L\cap\Delta^5$  is a triangle  which is given by the equations in~\eqref{4} and by the inequalities
\[
x_4\geqslant\frac{1}{9},\quad x_5\geqslant\frac{1}{9},\quad x_4+x_5\leqslant\frac{5}{9}. 
\]
\end{lem}

\begin{cor}\label{granicaP}
1) $\partial P = I_0\cup I_1\cup I_2$, where:
\[
\begin{split}
I_0 :  \; &  x_5=\frac{1}{9},\quad \frac{1}{9}\leqslant x_4\leqslant\frac{4}{9}, \quad  x_0=0,\\
  &  x_1=-\frac{1}{9}+x_4,\quad x_2=\frac{4}{9}-x_4,\quad x_3=\frac{5}{9}-x_4. 
\end{split}
\]
\[
\begin{split}
I_1 : \; &  x_4=\frac{1}{9},\quad \frac{1}{9}\leqslant x_5\leqslant\frac{4}{9},\quad x_0=-\frac{1}{9}+x_5,\\
 &  x_1=0,\quad x_2=\frac{4}{9}-x_5,\quad x_3=\frac{5}{9}-x_5. 
\end{split}
\]
\[
\begin{split}
I_2 : \; &  x_4+x_5=\frac{5}{9},\quad \frac{1}{9}\leqslant x_k\leqslant\frac{4}{9},\;  k=4,5,\\
&  x_0=-\frac{1}{9}+x_5,\quad x_1=-\frac{1}{9}+x_4,\quad x_2=0,\quad x_3=\frac{1}{9}.
\end{split}
\]

2) The vertices  $X_{i,j},\; 0\leqslant i<j\leqslant2$, of the triangle  $P$ are:
\[
X_{0,1} = I_0\cap I_1 : x_4=x_5=\frac{1}{9},\quad x_0=x_1=0,\quad x_2=\frac{1}{3},\quad x_3=\frac{4}{9}.
\]
\[
X_{0,2} = I_0\cap I_2 : x_4=\frac{4}{9},\quad x_5=\frac{1}{9},\quad x_0=x_2=0,\quad x_1=\frac{1}{3},\quad x_3=\frac{1}{9}.
\]
\[
X_{1,2} = I_1\cap I_2 : x_4=\frac{1}{9},\quad x_5=\frac{4}{9},\quad x_0=\frac{1}{3},\quad x_1=x_2=0,\quad x_3=\frac{1}{9}.
\]
\end{cor}
\subsection{Explicit description of $M_{Q}^{7}$}

For the clearness we first  note:

\begin{itemize}

\item $M^{7}_{Q}$ is invariant for the canonical $T^6$-action on $\C P^{5}$.
\item The effective action of $T^5 = T^{6}/T^1$, where $T^1=\text{diag}(T^6)$ on $M_{Q}^{7}$ is not free.  The points ${\bf z}\in M_{Q}^{7}$ for which one of the coordinates $z_0, z_1, z_2$ is zero  have  the non-trivial stabilizer.  The image of these points by the  moment map $\tilde{\mu}$ is given by the edges and the vertices of the triangle $P$ which are  described  by Corollary~\ref{edges}, or equivalently by Corollary~\ref{granicaP}.
\item  Since for ${\bf z}\in M_{Q}^{7}$ we have $z_3, z_4, z_5\neq 0$,  it follows that $M_{Q}^{7}$ is an affine manifold. In, addition,  the considered $T^5$-action on 
$M_{Q}^{7} $ can be realized by one of the representations \[(t_1, \ldots, t_5) \to (t_1, \ldots, t_5, 1),\;\;   (t_1, \ldots , t_5) \to (t_1, \ldots, t_4, 1, t_5),
\]
\[ 
(t_1, \ldots , t_5) \to (t_1, t_2, t_3, 1, t_4, t_5).\]
\end{itemize}

Let $S^5 = \{ (z_0, z_1, z_2)\in \C ^3 :  |z_0|^2 + |z_1|^2 + |z_2|^2 = \frac{1}{3}\}$.  

Consider the map $f : S^{5} \to \C P^{5}$ defined by
\begin{equation}\label{S5embed}
f(z_0, z_1, z_2) = (z_0 : z_1 : z_2 : |z_3| : |z_4| : |z_5|),
\end{equation}
where
\[ 
|z_4|^2 = |z_1|^2 +\frac{1}{9}, \;\; |z_5|^2 = |z_0|^2  + \frac{1}{9},
\]
\[
|z_3|^2  = \frac{4}{9} - |z_0|^2 - |z_1|^2.
\]

\begin{lem}
The image of $f$ is homeomorphic to $S^5$ and it  belongs to $M_{Q}^{7}$.
\end{lem}

Consider the representation  $\rho _1 : T^2 \to T^6$ given by
\begin{equation}\label{rho2}
\rho _1( t_4, t_5)  = (1, 1, 1, 1, t _4, t_5),
\end{equation}
and the representation $\rho _2 : T^5\to T^6$ given by 
\[\rho _2(t_1,t_2, t_3, t_4, t_5) = (t_1, t_2, t_3, 1, t_4, t_5).
\]

\begin{prop}
  The map $h : S^{5}\times T^2 \to \C P^5$  defined by
\[
\begin{split}
h(z_0, z_1, z_2,  t _4, t _5)  & =  \rho _1( t_4, t _5)f(z_0: z_1: z_2)\\
 & = (z_0 :  z_1 :  z_2 :  |z_3| : |z_4|t_4 : |z_5|t_5)
\end{split}
\]
is equivariant for the  $T^5$ -action on $S^5\times T^2$ given by 
\[
(\tau _0,\tau _1, \tau _2, \tau _4, \tau _5)((z_0, z_1, z_2), ( t_4, t_5)) =  (\tau _0 z_0, \tau _1 z_1, \tau _2 z_2 , \tau _4 t_4, \tau _5t_5) 
\]
 and the $T^5$-action on $\C P^{5}$ given  by the composition of the representation $\rho _2$ and  the canonical $T^6$-action.

 The image of $h$ is $M^{7}_{Q}$.  
\end{prop} 

\begin{proof}
The map is obviously equivariant and its image belongs to $M_{Q}^{7}$. We prove that it is onto $M_{Q}^{7}$. Let ${\bf z} = (z_0:z_1:z_2 :z_3:z_4: z_5)\in M_{Q}^{7}$, then, since $z_3\neq 0$ and $\frac{|z_0|^2}{\|{\bf z}\|^{2}}+\frac{|z_1|^2}{\|{\bf z}\|^{2}} + \frac{|z_1|^2}{\|{\bf z}\|^{2}} =\frac{1}{3}$ it follows that ${\bf z}$ has the representative of the form  $(z_0, z_1, z_2, |z_3|, z_4, z_5)$, where $|z_0|^2 +|z_1|^2 +|z_2|^2=\frac{1}{3}$. Since $z_4, z_5\neq 0$ then  $t_4, t_5\in S^1$ are uniquely determined by $z_k=|z_k|t_k$, $k=4,5$.  It follows that $h(z_0, z_1, z_2, t_4, t_5) = {\bf z}$.
\end{proof}
 Altogether we obtain:
\begin{thm}\label{equiv7}
The space $M^{7}_{Q}$ is  equivariantly homeomorphic  to $S^{5}\times T^2$,  for the given free  $T^5$-actions on $M_{Q}^{7}$ and $S^{5}\times T^2$.
\end{thm}

\begin{proof}
The map $h : S^{5}\times T^2 \to M_{Q}^{7}$ is  a $T^5$-equivariant surjection. The condition 
 $h(z_0, z_1, z_2,t_4, t_5) =   h(z_{0}^{'}, z_{1}^{'}, z_{2}^{'},   t_{4,}^{'}, t _{5}^{'})$ implies that 
$|z_{3}^{'}| = \lambda |z_3|$ and, since $|z_3|, |z_{3}^{'}|\neq 0$ this gives $\lambda =1$.  
\end{proof}


Consider now the representation  $\rho _3 :T^3\to T^5\to T^6$ given by 
\begin{equation}\label{rho3}
\rho _3: (t_1, t_2, t_3) \to (t_1, t_1, t_1, t_2, t_3)\stackrel{\rho _2}{\to} (t_1, t_1, t_1, 1, t_2, t_3).
\end{equation}
 This representation  defines free $T^3$-action on $M_{Q}^{7}$.  Let $\tilde{M}_{Q}^{4} = M_{Q}^{7}/T^3$. We deduce from Theorem~\ref{equiv7}:

\begin{thm}\label{MQ4}
The space $\tilde{M}_{Q}^{4}$ is homeomorphic to $\C P^{2}$. 

\end{thm}

\begin{rem}\label{T4MQ7}
Note that the exterior power representation $T^4\to T^6$ gives free $T^3=T^4/\text{diag}(T^4)$-action on $M_{Q}^{7}$. 
For this action  in~\cite{BT1} we proved that $\C P^{5}/T^4 \cong S^2\ast \C P^2$, which implies that $M_{Q}^{7}/T^4 \cong \C P^2$.
\end{rem}

\begin{rem}\label{eq}
In the previous construction we could equivalently take the representations $T^2\to T^6$ given by $(t_3, t_4) \to (1,1,1, t_3,  t_4, 1)$ or $(t_3, t_5)\to (1,1,1, t_3, 1, t_5)$, and consequently the representations $T^5\to T^6$ given by $(t_1, t_2, t_3, t_4, t_5)\to (t_1,t_2, t_3, t_4, t_5,1)$ or  $(t_1, t_2, t_3, t_4, t_5)\to (t_1, t_2, t_3, t_4, 1, t_5)$.
\end{rem} 

\subsection{Dependence  of the construction on a point $Q$ from the chamber orbit}

We noted that the manifold $M_{Q}^{7}$ does not depend, up to diffeomorphism, on a point $Q$  from the orbit of the chamber $x_1+x_i<1$, $i = 2,3,4$.  We point here that still the embedding 
$S^5\to \C P^5$ and the appropriate representation $T^2\to T^6$ do depend on a choice of a chamber  from the given orbit. 

The representative points of the chambers from this orbit are: $(\frac{1}{3}, \frac{5}{9}, \frac{5}{9}, \frac{5}{9})$,  $(\frac{5}{9}, \frac{1}{3}, \frac{5}{9}, \frac{5}{9})$,  $(\frac{5}{9}, \frac{5}{9}, \frac{1}{3}, \frac{5}{9})$,  $(\frac{5}{9}, \frac{5}{9}, \frac{5}{9}, \frac{1}{3})$. 
For example, for $Q= (\frac{5}{9}, \frac{1}{3}, \frac{5}{9}, \frac{5}{9})$, following the above construction,  we obtain  the embedding $f : S^5\to \C P^5$ given by
\[
f(z_0, z_3, z_4) = (z_0:|z_1|:|z_2|:z_3 : z_4 :|z_5|),
\]
 the representation $\rho _{1}: T^2\to T^6$ given by
\[
\rho_1 (t_2, t_5) = (1,1, t_2, 1, 1, t_5),\]
and  the representation $\rho _2 : T^5\to T^6$ given by
\[
\rho _2(t_1, t_2, t_3, t_4, t_5) = (t_1, 1, t_2, t_3, t_4, t_5).
\]
According to Remark~\ref{eq}, we can consider two other equivalent representations $(t_1, t_2, t_3, t_4, t_5) \to (t_1, t_2, 1, t_3, t_4, t_5)$ or $(t_1, t_2, t_3, t_4, t_5) \to (t_1, t_2, t_3, t_4, t_5,1).$  

\begin{rem}
All these representations of $T^5\to T^6$ are different from those pointed in Remark~\ref{eq}.
\end{rem}

\subsection{Construction for the second orbit}

We consider now the chamber $x_1+x_i>1$, $i=2,3,4$ and the point $Q= (\frac{2}{3}, \frac{4}{9}, \frac{4}{9}, \frac{4}{9})$ from this chamber. Then,   $M_{Q}^{7}\subset \C P^5$ is given by the following system:
\[
3(|z_0|^2 + |z_1|^2 + |z_2|^2) = 2Z, \;\;\; 9(|z_0|^2 + |z_3|^2 + |z_4|^2) =4Z,
\]
\[
9(|z_1|^2 + |z_3|^2 + |z_5|^2) =4Z, \;\; 9(|z_2|^2 + |z_4|^2 + |z_5|^2) =4Z,
\]
 for   $Z= \sum\limits _{i=0}^{5}|z_i|^2$ . This implies that 
\[
|z_0|^2 + |z_4|^2 = |z_1|^2 + |z_5|^2, \;\; 	|z_0|^2 + |z_3|^2 = |z_2|^2 + |z_5|^2,
\]
\[ |z_1|^2 + |z_3|^2 = |z_2|^2 + |z_4|^2,
\]
that is 
\begin{equation}\label{z4second}
|z_4|^2 =|z_1|^2 -|z_2|^2 + |z_3|^2
\end{equation}
 and
\begin{equation}\label{z5second}
|z_5|^2 = |z_0|^2-|z_2|^2 +|z_3|^2.
\end{equation}
Together with the first equation we obtain
\[
|z_2|^2 = \frac{1}{5}(|z_0|^2 +|z_1|^2 +6|z_3|^2).
\]
This gives that
\[
|z_0|^2 =\frac{1}{3}(|z_3|^2 + |z_4|^2 +4 |z_5|^2), \; |z_1|^2= \frac{1}{3}(|z_3|^2+4|z_4|^2+|z_5|^2),
\]
which implies 
\[
|z_2|^2 = \frac{1}{3}(4|z_3|^2 + |z_4|^2 +|z_5|^2).
\]
Comparing with~\eqref{z3},~\eqref{z4z5} we see that $M_{Q}^7$ can be obtained from the case previously considered  by the permutation $z_0\to z_3, z_1\to z_4, z_2\to z_5$.   Therefore, all previously said holds for the points of the chambers from this orbit.

Altogether we proved:

\begin{thm}\label{mainCP5}
The manifold $M_{Q}^{7}$ is homeomorphic to $S^5\times T^2$ for any regular value $Q\in \Delta _{4,2}$ of the standard moment map $\tilde{\mu} : \C P^{5}\to \Delta _{4,2}$.
\end{thm}

\section{The canonical $T^4$-action on $G_{4,2}$}

For the  Grassmann manifold $G_{4,2}$ , the image of the  Pl\"ucker embedding   $p : G_{4,2}\to \C P^5$ is   the hypersurface $z_0z_5 +z_2z_3= z_1z_4$.  The torus $T^4$ acts canonically on $G_{4,2}$, where $S^1=\text{diag}( T^4)$ fixes every point, so the torus $T^{3} = T^4/S^1$ acts effectively.
The moment map $\mu : G_{4,2}\to \Delta _{4,2}$ is given by
\[
\mu (L)  = \frac{1}{\sum\limits_{1\leq i<j\leq 4} |P^{I}(L)|^2}\sum\limits_{1\leq i<j\leq 4}|P^{ij}(L)|^2\Lambda _{ij}.
\]

Then  $M^{5}_{\bf x} = \mu ^{-1}({\bf x})$ is a smooth manifold for any ${\bf x}\in \stackrel{\circ}{\Delta} _{4,2}$ which belongs to a chamber of maximal dimension $3$. It holds $M_{\bf x}^{5} = M_{\bf x}^{7}\cap \{z_0z_5 +z_2z_3= z_1z_4\}$. Therefore, from Theorem~\ref{mainCP5} it follows:

\begin{cor}
 The manifolds   $M^{5}_{\bf x}$ and $M^{5}_{\bf y}$ are diffeomorphic for all ${\bf x}, {\bf y}$ which belong to  the chambers of maximal dimension.
\end{cor}

We determine the  homeomorphism  type of $M_{\bf x}^{5}$. It is enough to consider the point 
$Q = (\frac{1}{3}, \frac{5}{9}, \frac{5}{9}, \frac{5}{9})$ and we  denote by $M_{Q}^{5}$  the   level set at  $Q$ by the moment map $\mu$.

Since $\mu  = A\circ \tilde{\mu}$ we first  describe the projection $P^{'} = \tilde{\mu}(M_{Q}^{5})$.

\begin{lem}\label{P'}
The image $P^{'} \subset P\subset \Delta ^{5}$ is  a curve
\begin{equation}\label{curve}
\sqrt{x_0(x_0+\frac{1}{9})} + \sqrt{(\frac{1}{3} - x_0 -x_1)(\frac{4}{9} - x_0 - x_1)} = \sqrt{x_{1}(x_1 +\frac{1}{9})},
\end{equation}
whose intersection points  $X_i$ with the edges  $I_i$ of the triangle $P$ are:\\
$X_0 = (0, 1/6, 1/6, 5/18, 5/18, 1/9)$, $ X_1 = (1/6, 0,  1/6, 5/18,  1/9, 5/18),$ \\ $X_2 = (1/6, 1/6, 0, 1/9, 5/18 , 5/18)$.
\end{lem}

\begin{proof}  Since $\tilde{\mu}(z_0: z_1:z_3:z_4:z_5) = \tilde{\mu}(|z_0|: |z_1|: |z_2|:|z_3|: |z_4|: |z_5|)$,  it follows that $P^{'}$ is given by those points 
$x_i=\frac{|z_i|^{2}}{\|\bf{z}\|^{2}}$ for which $|z_0||z_5| + |z_3|z_4|=|z_1||z_5|$ which together with  Lemma~\ref{triangle} give~\eqref{curve}.  The intersection of $P^{'}$ with $I_{0}$ is given by those points from $I_0$ for which $x_2x_3=x_1x_4$. By Corollary~\ref{edges},  this is the same as  $(\frac{1}{3}-x_1)(\frac{4}{9}-x_1) = x_1(\frac{1}{9}+x_1)$,  which implies $x_1=\frac{1}{6}$,  so  we get the point $X_0$, and similarly  for the others.
\end{proof}

The level sets $M_{Q, i}^{3} = \tilde{\mu}^{-1}(X_i)$, $i=0,1,2$ in $G_{4,2}$ are the tori  $T^3$, where $X_i$ are given by Lemma~\ref{P'}.

\begin{lem}
The $M_{Q, i}$ are
\[
M_{Q, 0} = T^{3}\cdot  (0: \frac{1}{\sqrt{6}}: \frac{1}{\sqrt{6}} : \sqrt{\frac{5}{18}}: \sqrt{\frac{5}{18}}: \frac{1}{3}), \;\; T^{3} = (1, \tau _1, \tau_2, \frac{1}{\tau _2}, \frac{1}{\tau _1}, \tau _3),\]
\[M_{Q, 1} = T^{3}\cdot   ( \frac{1}{\sqrt{6}}: 0: \frac{1}{\sqrt{6}} : \sqrt{\frac{5}{18}}: \frac{1}{3}: \sqrt{\frac{5}{18}}), \;\; T^{3} = (\tau _1, 1, \tau_2, \frac{1}{\tau _2}, \tau _3, \frac{1}{\tau _1}),\]
\[
M_{Q, 2} = T^{3}\cdot  (\frac{1}{\sqrt{6}}: \frac{1}{\sqrt{6}} :  0:  \frac{1}{3}: \sqrt{\frac{5}{18}}: \sqrt{\frac{5}{18}}), \;\; T^{3} = (\tau _1, \tau_2, 1, \tau _3, \frac{1}{\tau _2}, \frac{1}{\tau _1}).\]

\end{lem}

 As we already said, the torus $T^3=T^4/S^1$ acts freely on $M_{Q}^{5}$, where $S^1$ is the diagonal in $T^4$ and  the action of $T^4$ is given by the composition of the  second exterior power representation $T^4\to T^6$ and the standard $T^6$-action.

We first emphasize the following:

\begin{prop}\label{MQ54}
There is no representation $T^3\to T^6$ which, when  composed with the canonical $T^6$-action  realizes  free $T^3=T^4/S^1$-action on $M_{Q}^{5}$.
\end{prop} 

\begin{proof}
Let $T^3$ acts on $M_{Q}^{5}\subset G_{4,2}\subset \C P^5$ by a representation $\rho : T^3\to T^6$ given by $\rho = (\rho _0, \ldots , \rho _5)$. Then it holds $\rho _0\rho _5 = \rho _1\rho _4=\rho _2\rho _3$. Namely,  notice that   for $(z_0: \ldots :z_5) \in M_{Q}^{5}$  we have  $z_0z_5+z_2z_3=z_1z_4$, then $z_3, z_4, z_5\neq 0$ and $|z_i|^2+|z_j|^2\neq 0$ for $0\leq i<j\leq 2$.    Therefore, the equality $\rho _0\rho _5z_0z_5 + \rho _2\rho _3z_2z_3 = \rho _1\rho _4z_1z_4$ implies for $z_0=0$ that 
\begin{equation}
\rho _2\rho _3=\rho _1\rho _4,  
\end{equation}
 since in this case $z_i\neq 0$ for $i\neq 0$. Similarly, for $z_2=0$ we obtain 
\begin{equation}\label{rhos}
\rho _0\rho _5=\rho _1\rho _4.
\end{equation}
Let  $\rho _{k}({\bf x}) = e^{i\langle  \Lambda _k, {\bf x}\rangle}$, $0\leq k\leq 5$, ${\bf x}\in \R^3 = \text{Lie}(T^3)$,   Then from~\eqref{rhos} it follows   $e^{i\langle \Lambda _2+\Lambda _3 -\Lambda _0-\Lambda _5, {\bf x}\rangle} =1$ for all ${\bf x}\in \R^3$. Because of continuity we have $\langle \Lambda _2+\Lambda _3 -\Lambda _0 - \Lambda _5 , {\bf x}\rangle =2k_0\pi$,  for all ${\bf x}\in \R^3$ and $k_0\in \Z$. For $k_0\neq 0$ this is $2$-dimensional plane, so $k_0=0$ and   $ \Lambda _2 + \Lambda _3 -\Lambda _0 -\Lambda _5=0$.  In the same way  $\Lambda _2+\Lambda _3 - \Lambda_1-\Lambda _4=0$. Without loss of generality we can take that  $\Lambda _0$, $\Lambda _1, \Lambda _2$ are linearly independent.   Let $\Lambda _3 = a\Lambda _0 + b\Lambda _1 + c\Lambda _2$ for $a, b, c\in \Z$.  Then in this basis the representation $\rho$ writes as 
\[
\rho (t_1, t_2, t_3) = (t_1, t_2, t_3, t_{1}^{a}t_{2}^{b}t_{3}^{c}, t_{1}^{a-1}t_{2}^{b}t_{3}^{c+1}, t_{1}^{a}t_{2}^{b-1}t_{3}^{c+1}).
\] 

For a point $(z_0:\ldots :z_5)\in M_{Q}^{5}$ such that $z_0, z_1, z_2\neq 0$,  the assumption that $T^4$-action and the action given by $\rho$ coincide, implies that 
\[
\tau _1\tau _2 = \lambda t_1,\;  \tau _1\tau_3 =\lambda t_2,\;  \tau _1\tau _4=\lambda t_3,\;  \tau _2\tau _3 = \lambda t_{1}^{a}t_{2}^{b}t_{3}^{c},\]
\[  \tau _{2}\tau _{4} = \lambda t_{1}^{a-1}t_{2}^{b}t_{3}^{c+1}, \;   \tau _3\tau _4 = \lambda t_{1}^{a}t_{2}^{b-1}t_{3}^{c+1}\]  for $\lambda \in S^1$ and for all $(\tau _1, \tau _2, \tau _3, \tau _4)\in T^4$ and all  $(t_1, t_2,t_3)\in T^3$. 
Thus, we obtain  
\[
\tau _1 = t_{1}^{-a}t_{2}^{1-b}t_{3}^{-c} \tau _2 = t_{1}^{1-a}t_{2}^{-b}t_{3}^{-c}\tau _2,
\]
which gives $t_1=t_2$ 

\end{proof}

This also gives, compare to  Theorem~\ref{MQ4} and Remark~\ref{T4MQ7}:

\begin{cor}
The  free $T^3=T^{4}/\text{diag}(T^4)$-action on $M_{Q}^{7}$ cannot be realized by the composition of the representation $T^3\to T^6$ and the canonical $T^6$-action on $M_{Q}^{7}$.
\end{cor}

\begin{rem}\label{T3free}
Note that there are representations $T^3\to T^6$,  which when  composed with the canonical $T^6$-action produce free $T^3$-action on $M_{Q}^{5}$. It is easy to verify that such representation would be $(t_1, t_2, t_3, 1, \frac{t_{3}}{t_2}, \frac{t_{3}}{t_1})$.
\end{rem} 

\subsection{The structure of   the principal fiber bundle $M_{Q}^{5}\to \C P^{1}$}

The torus $T^3 = T^{4}/\text{diag}(T^4)$ acts freely on $M_{Q}^{5}$ and it induces the principal fiber bundle $M^{5}_{Q} \stackrel{\pi}{\to} M_{Q}^{5}/T^3$  with a fiber $T^3$. In~\cite{BT1}, it is proved that $G_{4,2}/T^3\cong S^3\ast \C P^1\cong S^5$ and, in particular, that  $\mu ^{-1}({\bf x})/T^3 \cong \C P^1$ for any ${\bf x}\in \stackrel{\circ}{\Delta} _{4,2}$. 

We first examine the  structure of this principal bundle.


 The Pl\"ucker charts on $G_{4,2}$ are given by $M_{ij} =\{ L\in G_{4,2} | P^{ij}(L)\neq 0\}$, $1\leq i <j\leq 4$. From the definition of $M_{Q}^5$ we obtain that
\[
M_{Q}^{5}\not\subset M_{1i}  , \; i=2,3,4\; \text{and}\; M_{Q}^{5}\subset M_{ij}, \; 2\leq i<j\leq 4.
\]
In particular, we deduce in another way:
\begin{cor}
$M_{Q}^{5}$ is  a smooth  affine manifold.
\end{cor}
The strata in $G_{4,2}$  with the free $T^3$-action are the main stratum $W_4$ and  $W_{ij} = \{L\in G_{4,2} | P^{ij}(L) = 0\}$, $1\leq i<j\leq 4$. It follows that 
\[
M_{Q}^{5} \cap  W \neq \emptyset, \;  W= W_4,  W_{12}, W_{13}, W_{14}, \;\; M_{Q}^{5}\cap W_{ij} =\emptyset, \;  2\leq i<j\leq 4.
\]

Let $A =\{(0:1), (1:0), (1:1)\}$ and $\C P^{1}_{A} = \C P^{1}\setminus A$. For any $(c:c^{'})\in \C P^{1}_{A}$ the set $\pi^{-1}((c:c^{'}))$ belongs, according to~\cite{BT1},  to the main stratum $W_4$, that is
\[
\pi ^{-1}(\C P^{1}_{A})\cong T^{3}\times \C P^{1}_{A}.
\]
Let $B_{0} = (0:1), B_1=(1:1)$, $B_{\infty} =(1:0)$ and $\C P ^{1}_{B_i} = \C P^{1}\setminus \{B_i\}$, $i=0,1,\infty$.

\begin{prop}\label{bezstrata}
 The following holds:
\[ 
M_{Q}^{5}\setminus( M_{Q}^{5}\cap W_{12})  \cong T^3\times \C P^{1}_{B_{\infty}},\]
\[ M_{Q}^{5}\setminus( M_{Q}^{5}\cap W_{13})  \cong T^3\times \C P^{1}_{B_{0}}, \] 
\[M_{Q}^{5}\setminus( M_{Q}^{5}\cap W_{14})  \cong T^3\times \C P^{1}_{B_{1}}. 
\]
\end{prop}
\begin{proof}
Consider the chart  $M_{23}$ on $G_{4,2}$ with the local coordinates $a_1, a_2, a_3, a_4$. Then any $L\in G_{4,2}$  represents in $M_{23}$   by
\[
A_{L} =\left (
\begin{array}{cc}
a_1 & a_3\\
1 & 0\\
0 & 1\\
a_2 & a_4
\end{array}\right ).
\]
For the points from $M_{Q}^{5}$ we have that $|a_i|^2+|a_{j}|^{2}\neq 0$, $1\leq i<j\leq 4$ and $a_2, a_4\neq 0$. In addition, 
\begin{equation}\label{jednacine}
1= \frac{1}{3} (|a_3|^2 + |a_1|^2 +4 |a_1a_4-a_2a_3|^2), 
\end{equation}
\[
 |a_4|^2 = \frac{1}{3} (|a_3|^2 + 4 |a_1|^2 + |a_1a_4-a_2a_3|^2),
\]
\[
|a_2|^2 = \frac{1}{3} (4|a_3|^2 + |a_1|^2 + |a_1a_4-a_2a_3|^2).
\]
The  points of the space $M_{Q}^{5}\setminus( M_{Q}^{5}\cap W_{12})$  satisfy additional condition $a_3\neq 0$.
 Consider the map $f : M_{Q}^{5}\setminus( M_{Q}^{5}\cap W_{14}) \to \C P^{1}_{B_{1}}\times T^3$ defined  by
\[
f(a_1, a_2, a_3, a_4) = ((a_1a_4:a_2a_3), (t_2, t_3, t_4)), 
\]
where $a_k = |a_k|t_k$, $k=1,2,3,4$.  Since   $f(a_1, a_2, a_3, a_4)= f(a_{1}^{'}, a_{2}^{'}, a_{3}^{'}, a_{4}^{'})$ implies
$t_i^{'} = t_i$ for $i=2,3, 4$ and $a_{1}^{'}a_{4}^{'} = c a_1a_4$,  $a_{2}^{'}a_{3}^{'} = c a_2a_3$ for $c = |c|\lambda \in \C$, $c\neq 0, \lambda \in S^1$, the map $f$ is homeomorphism.  Hence
$t_{2}^{'}t_{3}^{'} = \lambda t_{2}t_{3}$,  that is $\lambda =1$. Since  $a_1\neq 0$ iff $a_{1}^{'}\neq 0$, we have that  $t_{1}^{'} = t_1$ for $a_1\neq 0$.  Now,  from~\eqref{jednacine} we deduce that $a_i=a_{i}^{'}$, $i=1,2,3,4$. 
\end{proof}

Consider now the standard charts on $\C P^{1}$ given by $M_{0} = \{(1:c) | c\in \C\}$ and $M_{1} = \{(c:1) | c\in \C \}$ and the projection $\pi : M_{Q}^{5}\to \C P^{1}$.

\begin{lem}\label{CP1}
It holds 
\begin{equation}
\pi ^{-1}(M_0) = M_{Q}^{5}\setminus( M_{Q}^{5}\cap W_{13}), \;\;  \pi ^{-1}(M_1) =M_{Q}^{5}\setminus( M_{Q}^{5}\cap W_{12}) .
\end{equation}
\end{lem}

\begin{proof}
It follows immediately from  the proof of Proposition~\ref{bezstrata}.
\end{proof}

\begin{thm}
The structure of the principal fiber bundle $M_{Q}^{5}\to \C P^1$ can be described by two  charts $M_{0}$, $M_{1}$  for  $\C P^{1}$   with  the transition function 
\[h_{01} : \pi ^{-1}(M_0 \cap M_1) \to \pi ^{-1}(M_0\cap M_1), \\
h_{01}((c:c^{'}), (t_{1}^{0}, t_{2}^{0}, t_{3}^{0})) = ((c:c^{'}), (t_{1}^{1}, t_{2}^{1}, t_{3}^{1})), 
\]
where 
\begin{equation}\label{transition}
t_{1}^{1}= \frac{t_{1}^{0}t_{2}^{0}}{t_{3}^{0}},\; t_{2}^{1} =t_{1}^{0}, \; t_{3}^{1} = t_{3}^{0}.
\end{equation}
\end{thm}

 \begin{proof}
Since  $\pi ^{-1}(M_0 \cap M_1) = M_{Q}^{5}\setminus ( M_{Q}^{5}\cap W_{13} \cup  M_{Q}^{5}\cap W_{12})$ it follows from Proposition~\ref{bezstrata} that $\pi ^{-1}(M_0\cap M_1)$ is  homeomorphic to  $\C P^1\setminus \{B_{0}, B_{\infty}\} \times T^3$. The canonical action of  $T^4$ on $G_{4,2}$ produces the action of $T^3$ on the coordinates  of the chart $M_{23}\cong \C ^4$ defined by the coordinate wise action of  $(\tau _1, \tau _2, \tau _3, \frac{\tau _2\tau _{3}}{\tau _1})$.
In the chart $\pi ^{-1}(M_0)$ we have that $a_3\neq 0$,  so the torus $T_{0}^{3} = \{( \tau_2, \tau_3, \frac{\tau_2\tau_3}{\tau_1} )\}$  acts freely on $\pi ^{-1}(M_0)$  and 
$\pi ^{-1}(M_0) \cong \C P^{1}_{B_1}\times T^{3}_{0}$. In the chart $\pi ^{-1}(M_1)$ we have that $a_1\neq 0$ and the torus $T^{3}_{1} = \{( \tau_1, \tau_2, \frac{\tau _2\tau _3}{\tau _1})\}$ acts freely on $\pi ^{-1}(M_1)$  and  $\pi ^{-1}(M_1) \cong \C P^{1}_{B_{\infty}}\times T^{3}_{1}$. Thus,  the transition map is given by~\eqref{transition}.
\end{proof}

\section{Explicit description of $M_{Q}^{5}$}

Using the results on description of $M_{Q}^{7}$,  we obtain explicit description of $M_{Q}^{5}$ stated by Theorem~\ref{mainnew}, which is  one of our  main results.

 Since $M_{Q}^{5} = M_{Q}^{7}\cap G_{4,2}$,  from~\eqref{definingequations} we deduce:
\begin{lem}\label{MQ5affine}
The manifold $M_{Q}^{5}$ is an affine manifold in $\C ^5$ given by the points $(z_0, z_1, z_2, |z_3|, z_4, z_5)$ such that $z_0z_5 + z_{2}|z_3| =z_1z_4$ and  $|z_0|^2+|z_1|^2+|z_2|^2=\frac{1}{3}$.
\end{lem}
\begin{proof}
Any point  ${\bf z} = (z_0:\ldots :z_5)\in M_{Q}^{5}$ has an unique representative of the form $(z_{0}^{'}, z_{1}^{'}, z_{2}^{'}, |z_{3}^{'}|, z_{4}^{'}, z_{5}^{'})$ satisfying the given conditions. 
More precisely,  for ${\bf z}$ we have $z_3 = |z_3|e^{i\psi _{3}}\neq 0$ and $\frac{1}{\|{\bf z}\|^{2}}(|z_0|^2+|z_1|^2+|z_2|^2)=\frac{1}{3}$, so the required  representative is given by 
$\frac{e^{-i\psi _3}}{\|{\bf z}\|^2}(z_0:\ldots :z_5)$.
\end{proof}

Let us consider the sphere $S^{5}\subset M_{Q}^{7}$ given by the embedding~\eqref{S5embed}.

\begin{lem}\label{intersection}
The intersection of $M_{Q}^{5}$ with the sphere $S^5$ is  the closed  $3$-dimensional surface $M^{3} \subset \C P^5$ given by the points 
\[ 
(z_0 :  z_1:  z_2 : |z_3| :  |z_4| : |z_5|) \in \C P^5,
\]
where
\begin{equation}\label{2345}
|z_4|^2 = |z_1|^2 +\frac{1}{9}, \;\; |z_5|^2 = |z_0|^2  + \frac{1}{9}, \;\; |z_2|^2 = \frac{1}{3} - |z_0|^2 -|z_1|^2, 
\end{equation}
\[
|z_3|^2  = |z_2|^2 + \frac{1}{9}, 
\]
which satisfy
 \begin{equation}\label{surface}
z_0 \sqrt{|z_0|^2 +\frac{1}{9}} + z_2\sqrt{|z_2|^2 + \frac{1}{9} }
 = z_1 \sqrt{|z_1|^2 +\frac{1}{9}}. 
\end{equation}
\end{lem}

\begin{proof}
The point from $M_{Q}^{5}\cap S^5$ satisfy the equation  $z_0|z_5| +z_2|z_3| = z_1|z_4| $, which writes as~\eqref{surface}.
\end{proof}

\subsection{$M_{Q}^{5}$ as $T^3$-orbit of $2$-dimensional surface}
Using a free   $T^3$-action  $M_{Q}^{5}$ we first examine, following the  results on $M_{Q}^7$,   the description of   $M_{Q}^{5}$ as a $T^3$-orbit of  the $2$-dimensional surface $M^2$  in $S^5$. We prove that $M_{Q}^{5}$ is equivariantly homeomorphic to the quotient space of  $M^2\times T^3$  for the corresponding free $T^3$-actions. We show in section 7 below that $M_{Q}^{5}$ is homeomorphic to $S^3\times T^2$.

Let $M^{2}\subset M^{3}$ be a $2$-dimensional  surface defined by the condition $\phi _2=0$, that is 
\begin{equation}\label{M2}
M^{2} = \{(z_0: z_1: |z_2|:| z_3| : |z_4| :|z_5|)\in M^{3}\} .
\end{equation}
Let $\hat{M}^2 \subset \C ^2$ be a closed surface given by~\eqref{surface} for $\phi _2=0$, that is 
\[
z_0 \sqrt{|z_0|^2 +\frac{1}{9}} + \sqrt{\frac{1}{3}-|z_0|^2 -|z_1|^2}\sqrt{\frac{4}{9} -|z_0|^2-|z_1|^2}
 = z_1 \sqrt{|z_1|^2 +\frac{1}{9}}.
\]

 The map  $r : \hat{M}^{2} \to \C P^5$, $ (z_0, z_1) = (z_0 : z_1 :|z_2| : |z_3| : |z_4|: |z_5|)$ gives a homeomorphism between  $\hat{M}^{2}$ and  $M^{2}$.


The manifold $M_{Q}^{7}$ is invariant for any   $T^k\subset T^6 $ action, while $M_{Q}^{5}$ is invariant for $T^{4}\subset T^6$-action for the $T^4$-representation    given by the second exterior power. In the context of Proposition~\ref{MQ54} and Remark~\ref{T3free} we  note the following:

\begin{lem}
The manifolds $M_{Q}^{7}\subset \C P^{5}$ and $M_{Q}^{5}\subset \C P^{5}$ are invariant for the $T^{3}$-action defined by the representation   $T^3\to  T^6$  given by
\begin{equation}\label{ro}
\rho (t_1, t_2, t_3 ) = (t_1,  t_2,  t_3, 1, \frac{t_3}{t_{2}}, \frac{t_3}{t_{1}}),
\end{equation}
 and the canonical  $T^6$-action  on $\C P^5$.
\end{lem}


Consider the surface $M^{2}$ defined by~\eqref{surface}. 

\begin{prop}\label{MQ5}
The map $F : M^2\times  T^3 \to \C P^5$ defined by
\[
\begin{split}
F((z_0: z_1 : |z_2| : |z_3|:  |z_4|, |z_5|), (t_1,t_2, t_3) )& = \rho (t_1, t_2, t_3) (z_0: z_1 : |z_2| : |z_3|:  |z_4| : |z_5|) \\
 & =
(t_1z_0 : t_2z_1 : t_3  |z_2| :  |z_3| : \frac{t_3}{t_2}|z_4|: \frac{t_3}{t_1}|z_5|)  
\end{split}
\]
is $T^3$-equivariant   for the $T^3$-action on $M^{2}\times  T^3$ given by 
\[
\begin{aligned}
(\tau_1, \tau _2, \tau_3) ((z_0: z_1:|z_2|:|z_3|:|z_4|:|z_5|), (t_1,  t_2, t_3))\\
=((z_0:z_1:|z_2|:|z_3|:|z_4|:|z_5|), (\tau _1t_1,  \tau _2 t_2, \tau _3t_3))
\end{aligned}
\]
 and the  $T^3$-action on $\C P^{5}$ is given by the representation $\rho$.

The image of  $F$ is $M_{Q}^{5}$.

\end{prop}

\begin{proof}

The map $F$ is obviously equivariant for the given $T^3$-actions on $M^2\times T^3$ and $\C P^4$. We prove that its  image is $M_{Q}^{5}$. By Lemma~\ref{MQ5affine}, let $(z_0, z_1, z_2, |z_3|, z_4, z_5)\in M_{Q}^{5}$ and assume that $z_2\neq 0$, Then    $z_k = |z_k|e^{i\psi _k}$, $0\leq k\leq 5$ and take   $t_3 = e^{i\psi _2}$. Then put $t_1 = e^{i(\psi _2-\psi _5)}$, $t_2=e^{i(\psi _2-\psi _4)}$. In addition, let  $ z_{0}^{'} = |z_0|e^{i(\psi _0-\psi _2+\psi _5)}$, $ z_{1}^{'} = |z_1|e^{i(\psi _1-\psi _2+\psi _4)}$, $|z_{k}^{'}|=|z_k|$ for $2\leq k\leq 5$. 
Then 
\[
F((z_{0}^{'},z_{1}^{'}, |z_{2}|^{'},  |z_{3}^{'}|,  |z_{4}^{'}|, |z_{5}^{'}|), (t_1, t_2, t_3)) = (z_0, z_1, z_2, |z_3|, z_4, z_5).
\]
 The   point $(z_{0}^{'}, z_{1}^{'}, |z_{2}^{'}|, |z_{3}^{'}|, |z_{4}^{'}|, |z_{5}^{'}|)$ belongs to $M^2$ since 
\[
\begin{split}
z_{0}^{'}|z_{5}^{'}|+|z_{2}^{'}||z_{3}^{'}| & = \frac{1}{t_1}z_0\frac{t_1}{t_3}z_5 + \frac{1}{t_3}z_2|z_3|\\
&  = \frac{1}{t_3}z_1z_4 =\frac{1}{t_3}t_2z_{1}^{'}\frac{t_3}{t_2}|z_{4}^{'}| =
 z_{1}^{'}|z_{4}^{'}|.
\end{split}
\]
If $z_2=0$  we take $t_3=1$ and $t_1 = e^{-i\psi _5}$, $t_2=e^{-i\psi _4}$, then   $ z_{0}^{'} = |z_0|e^{i(\psi _0+\psi _5)}$, $ z_{1}^{'} = |z_1|e^{i(\psi _1+\psi _4)}$, $|z_{k}^{'}|=|z_k|$ for $2\leq k\leq 5$. We obtain \[
 F((z_{0}^{'},z_{1}^{'}, 0,  |z_{3}^{'}|,  |z_{4}^{'}|, |z_{5}^{'}|), (t_1, t_2, t_3)) = (z_0, z_1, 0, |z_3|, z_4, z_5)\] and 
\[
\begin{split}
z_{0}^{'}|z_{5}^{'}| & =\frac{1}{t_1}z_0\frac{t_1}{t_3}z_5 =\frac{1}{t_3}z_1z_4 \\
  & =  \frac{1}{t_3}t_2z_{1}^{'}\frac{t_3}{t_2}|z_{4}^{'}| =
 z_{1}^{'}|z_{4}^{'}|,
\end{split}
\] 
which confirms that $(z_{0}^{'}, z_{1}^{'}, 0, |z_{3}^{'}|, |z_{4}^{'}|, |z_{5}^{'}|)$ belongs to $M^2$.
\end{proof}

\begin{rem}
The given $T^3$-actions  on $M^{2}\times T^3$ and on  $M_{Q}^{5}$ are free.
\end{rem}

 \begin{cor}\label{equiv5}
The space $M^{5}_{Q}$ is equivariantly homeomorphic  to the quotient space  $\tilde{M}_{Q}^{5}$ of $M^{2}\times T^3$ by the equivalence relation defined by the map $F$.
\end{cor}

\begin{lem}\label{equiv23}
The equivalence relation on $M^{2}\times T^3$ defined by the map $F : M^{2}\times T^3 \to M_{Q}^{5}$ is given as follows: $(z_{0}^{'}, z_{1}^{'}, t_{1}^{'}, t_{2}^{'}, t_{3}^{'})\approx (z_0, z_1, t_1, t_2, t_3)$ if and only if
\begin{enumerate}
\item $z_0 \neq  z_1$ and 
\[
z _{k}^{'} = z_{k}, \; k=0,1, \;\; t_k=t_{k}^{'}, \; k=1,2,3
\]

\item  $z_0 =z_1$ and  $z_{0}^{'} =z_{1}^{'}$ and
\[
z_{0}^{'} = \frac{1}{\lambda}z_0, \;\; t_{k}^{'} =\lambda t_k, \; k=1,2,3, 
\]
for $\lambda \in S^1$. 
\end{enumerate}
\end{lem}
\begin{proof}
We first note that
\[
\rho (t_1,t_2, t_3)(M^{2}\setminus \{z_2=0\}) \cap M^{2} =\emptyset \; \text{for}\; (t_1, t_2, t_3)\neq (1,1,1).
\]
Precisely,  if for ${\bf z} = (z_0: z_1 : |z_2| : |z_3| : |z_4| : |z_5|)\in M^{2}\setminus \{z_2=0\}$  it holds $\rho (t_1, t_2, t_3)({\bf z}) \in M^2$ for some $(t_1, t_2, t_3)\neq (1,1,1)$,  then $t_3=1$ and $\frac{t_3}{t_2}= \frac{t_3}{t_1}=1$, that is $t_1=t_2=1$.
 Since the $T^3$-action on $M_{Q}^{5}$ given by the representation $\rho$ is free, we deduce that the map
\[
F : (M^{2}\setminus \{z_2=0\})\times T^3 \to M_{Q}^{5}\setminus \{z_2=0\}
\]
is a homeomorphism.

For $z_2=0$ we see from~\eqref{surface} that $z_0=z_1$. From~\eqref{2345} it follows that
\[
|z_0| = \frac{1}{\sqrt{6}}, \; |z_3| = \frac{1}{3}, \; |z_4| = |z_5| =  \frac{\sqrt{5}}{\sqrt{18}}, 
\]
that is,  we obtain the circle in $M^{2}$ given by the points 
\[
(\frac{e^{i\psi _{0}}}{\sqrt{6}}, \frac{e^{i\psi _{0}}}{\sqrt{6}},  0 , \frac{1}{3}, \frac{\sqrt{5}}{\sqrt{18}} , \frac{\sqrt{5}}{\sqrt{18}}).
\]

It follows that $\frac{t_{3}^{'}}{t_{1}^{'}} = \frac{t_3}{t_1}$ and  $\frac{t_{3}^{'}}{t_{1}^{'}} = \frac{t_3}{t_1}$, which gives  $\frac{t_{1}^{'}}{t_{1}} = \frac{t_{2}^{'}}{t_2} =\lambda$, that is $t_{k}^{'} = \lambda t_k$, $k=1,2, 3$.  Since $t_{1}^{'}z_{0}^{'}= t_1z_0$ it follows that $z_{0}^{'}=\frac
{1}{\lambda}z_0$.

\end{proof}

\begin{rem}
Since the map $F: M^{2}\times T^3\to M_{Q}^{5}$ is equivariant  for the representation $\rho$, the $T^3$-action on $M^{2}\times T^3$ induces $T^3$-action on $\tilde{M}^{5}_{Q} = ( M^2\times T^3)/_{F}$.
\end{rem}

\subsection{$M_{Q}^{5}$ as $T^2$-orbit of  a $3$-dimensional manifold}
 
Consider the space $M^{3}$ given by Lemma~\ref{intersection},  that is by 
\[ \{(z_0:z_1:z_2: |z_3|:|z_4|:|z_5|)\in \C P^5\}, \] such that~\eqref{2345} and~\eqref{surface} are satisfied.

 Let the representation $\hat{\rho} : T^2 \to T^6$ be given by $\hat{\rho} (t_1, t_2) = (t_1, t_2, 1, 1, \frac{1}{t_2}, \frac{1}{t_1})$.

\begin{prop}\label{M3MQ5}
The map $G : M^3\times T^2 \to \C P^5$ defined by 
\[
G((z_0:z_1:z_2: |z_3|:|z_4|:|z_5), (t_1, t_2)) = (t_1z_0:t_2z_1 : z_2: |z_3|:\frac{1}{t_1}|z_4|:\frac{1}{t_2}|z_5|)
\]
is $T^2$-equivariant for $T^2$-action on $M^{3}\times T^2$ given by 
\[(\tau _1, \tau _2)\cdot ((z_0:z_1:z_2:|z_3|:|z_4|:|z_5), (t_1, t_2)) =
((z_0:z_1:z_2:|z_3|:|z_4|:|z_5), (\tau _1 t_1, \tau _2 t_2))
\]
 and the $T^2$-action on $\C P^5$ defined  by the composition of the representation $\hat{\rho}$ and the canonical  $T^6$-action on $\C P^5$.

The image of $G$ is $M_{Q}^{5}$ and $G$ is one-to-one.
\end{prop}

\begin{proof}
It checks directly that the map $G$ is $T^2$-equivariant for the given $T^2$-actions. It  is onto $M_{Q}^{5}$. For  $(z_0, z_1, z_2, |z_3|, z_4, z_5)\in M_{Q}^{5}$, where 
$z_k= |z_k|e^{i\psi _k}$, $0\leq k\leq 5$,   take $t_{1} =e^{-i\psi _5}$, $t_{2}=e^{-i\psi _4}$ and $|z_{3}^{'}| = |z_3|$, $z_{2}^{'} = z_2$, and $z_{0}^{'} = |z_0|e^{i(\psi _0+\psi _5)}$, $z_{1}^{'} = |z_1|e^{i(\psi _1 +\psi _4)}$. Then 
\[
G((z_{0}^{'} : z_{1}^{'}: z_{2}^{'}: |z_{3}^{'}|: |z_{4}^{'}|: |z_{5}^{'}|), (t_{1}, t_{2}))= (z_0 : z_1: z_2 : |z_3|:  z_4 : z_5).
\]
In addition, 
\[
\begin{split}
z_{0}^{'}|z_{5}^{'}| + z_{2}^{'}|z_{3}^{'}|  & = t_1z_0 \frac{1}{t_1}z_5 + z_2|z_3| \\
 &  =z_1|z_4| = t_{2}z_{1}^{'}\frac{1}{t_2}|z_{4}^{'}| = z_{1}^{'}|z_{4}^{'}|, 
\end{split}
\]
which means that $(z_{0}^{'} : z_{1}^{'}: z_{2}^{'}: |z_{3}^{'}|: |z_{4}^{'}|: |z_{5}^{'}|)\in M^{3}$. That $G$ is one-to-one follows from the fact that $|z_4|, |z_5|\neq 0$. 
\end{proof}

Therefore, we deduce:

\begin{thm}
The manifold $M_{Q}^{5}$ is homeomorphic to $M^3\times T^2$.
\end{thm}

\section{Relations between $M^2$, $M^3$ and $\C P^1$ towards $M_{Q}^5$}

We determine the relations among the spaces $M^2$, $M^3$ and $\C P^1$,  which lead to the   explicit description of $M_{Q}^{5}$ formulated by Theorem~\ref{mainnew}.

\subsection{$M^2$ towards  $\C P^1$ and $M_{Q}^{5}$}

Recall that the space $M^{2}$ belongs to the  intersection   of $M_{Q}^{5}$ and $S^{5}$  in $M_{Q}^{7}$ and  it is given  by the points $(z_0 : z_1 : |z_2| : |z_3| : |z_4| : |z_5|)\in \C P^5$ such that
\[
z_0|z_5| + |z_2||z_3| = z_1|z_4|,
\]
where $|z_k|$ is given by~\eqref{2345}  for $k=2,3,4,5$.

Since $M_{Q}^{5}\subset M_{23}$ it follows that $M^{2}\subset M_{23}$ and its coordinates $(a_1, a_2, a_3, a_4)$ in this chart are:
\[
a_1=\frac{P^{13}}{P^{23}}, \; a_2 = -\frac{P^{34}}{P^{23}}, \; a_3=-\frac{P^{12}}{P^{23}}, \; a_4 = \frac{P^{24}}{P^{23}}.
\]
We  obtain  that
\[
a_1 =\frac{z_1}{|z_3|}, \; a_2 = -\frac{|z_5|}{|z_3|}, \; a_3 = -\frac{z_0}{|z_3|},\; a_4 = \frac{|z_4|}{|z_3|}.
\]
The quotient $M_{Q}^{5}/T^4$ for the canonical action of $T^4$ on $G_{4,2}$ is $\C P^1$ and, according to~\cite{BT1},  it can be described  by the projection $M_{Q}^{5}\to \C P^1$ defined  by $(a_1,a_2, a_3, a_4)\to (a_1a_4 : a_2a_3)$. In this way we obtain the projection $p : M^{2} \to \C P^1$  defined by
\begin{equation}\label{M2CP1}
p(z_0:z_1: |z_2|: |z_3|: |z_4|: |z_5|) = (z_1|z_4| : z_0|z_5|).
\end{equation}

Recall that the circle $S^1\subset M^2$ defined by $z_0=z_1$ is given by the parametrization $(\frac{e^{i\psi _{0}}}{\sqrt{6}}, \frac{e^{i\psi _{0}}}{\sqrt{6}},  0 , \frac{1}{3}, \frac{\sqrt{5}}{\sqrt{18}} , \frac{\sqrt{5}}{\sqrt{18}})$.
\begin{thm}\label{M2ramified}
The projection $p : M^{2}\setminus S^{1}\to \C P^1\setminus\{(1:1)\}$ is a  homeomorphism.
\end{thm}
\begin{proof}
The condition $p({\bf z}) = p({\bf z^{'}})$ implies that $z_0|z_5| = re^{i\psi} z_{0}^{'}|z_{5}^{'}|$ and $z_1|z_4| = re^{i\psi}z_{1}^{'}|z_{4}^{'}|$. 
It gives  $\phi _k = \phi _{k}^{'} + \psi$ for  $k = 0, 1$, where $z_k= |z_k|e^{i\phi _k}$ and  $z_k^{'}= |z_k^{'}|e^{i\phi _k^{'}}$. In addition,  the points ${\bf z}$ and ${\bf z^{'}}$ belong to the same $T^{4}$-orbit in $M_{Q}^{5}$, which gives that $|z_k|=|z_{k}^{'}|$, $k=0,1,2,3,4,5$.

Assume that $z_0\neq  z_1$, $z_0, z_1\neq 0$.  Since 
\[
z_0|z_5| + |z_2||z_3|  = z_1|z_4|\]
\[
 re^{i\psi}z_{0}^{'}|z_{5}^{'}| + re^{i\psi}|z_{2}^{'}||z_{3}^{'}|  = re^{i\psi}z_{1}^{'}|z_{4}^{'}|,
\]
we obtain that 
\[
 re^{i\psi}|z_{2}^{'}||z_{3}^{'}|  - |z_{2}||z_{3}| = 0.
\]
Thus,
\[
re^{i\psi}  =1.
\]

\end{proof}

In addition we have:
\begin{lem}\label{circle}
The preimage $p^{-1}(1:1)$ of the projection $p: M^{2}\to \C P^1$ is the circle $S^1$.
\end{lem}
\begin{proof}
If $p({\bf z}) = (1:1)$ then $z_0|z_5| = z_1|z_4|$.  Thus,
\[
|z_0|^2|z_5|^2 = |z_1|^2|z_4|^2, \]
which by~\eqref{2345} gives
\[ \frac{1}{9}(|z_0|^2 -|z_1|^2) + |z_0|^4 - |z_1|^4 =0.\]

It implies that
\[
(|z_0|^2 -|z_1|^2)(\frac{1}{9} + |z_0|^2 + |z_1|^2) =0,  \; \text{so} \; |z_0|=|z_1|.
\]
Since $z_0|z_5| +|z_2||z_3|= z_1|z_4|$, we have that $|z_2|=0$, which is equivalent to $|z_0|^2 + |z_1|^2 =\frac{1}{3}$. It follows that
\[
|z_0| = |z_1|=\frac{1}{\sqrt{6}}, \; |z_3| = \frac{1}{\sqrt{3}},\;  |z_4| = |z_5| =\frac{\sqrt{5}}{\sqrt{18}}.
\]
Therefore, $\phi _{0} =\phi _{1}=\psi$  and 
\[
p^{-1}(1:1) = (\frac{1}{\sqrt{6}}e^{i\psi}:\frac{1}{\sqrt{6}}e^{i\psi} : 0 : \frac{1}{\sqrt{3}} : \frac{\sqrt{5}}{\sqrt{18}} : \frac{\sqrt{5}}{\sqrt{18}}).
\]
\end{proof}

\begin{cor} \label{cor1}
The sphere $\C P^1$ is homeomorphic to the quotient space $\tilde{M}^{2}$  of $M^{2}$ by  the equivalence relations defined by the projection $p$, that is
$(z_0^{'}, z_{1}^{'})\approx (z_0, z_1)$ if and only if 
\begin{itemize}
\item $z_k = z_{k}^{'}$, $k=0,1$  for  $z_0\neq z_1$, 
\item  $z_0=z_1$ and $z_{0}^{'} = z_{1}^{'}$.
\end{itemize}
\end{cor}

\begin{cor}\label{disc}
The space $M^{2} = \overline{M^{2}\setminus S^1}$ is homeomorphic to the closed disc $D^2$.
\end{cor}
From Corollary~\ref{equiv5}  and Lemma~\ref{equiv23} we deduce:

\begin{cor}\label{cor2}
The orbit space $\tilde{M}^{5}_{Q}/T^3$ by the action induced from the $T^3$-action on $M^{2}\times T^3$ is homeomorphic to $\C P ^1$.
\end{cor}

Let  $M_{2} = D^{2}\setminus S^1$, where $ S^1  = \{(\frac{e^{i\phi _0}}{\sqrt{6}}, \frac{e^{i\phi _0}}{\sqrt{6}})\in M^{2}\}$.   Lemma ~\ref{equiv23} gives:

\begin{cor}\label{tildeMproduct1}
\begin{itemize}
\item
The space $(M_{2}\times T^3) /\approx$ is homeomorphic to $M_{2}\times T^{3}$.
\item 
The space $(S^1\times T^3)/\approx$ is homeomorphic to $T^4/S^1$, where $S^1$ acts on $T^4$ by 
$\lambda \cdot  (t_{0}, t_{1}, t_{2}, t_3) = (\frac{1}{\lambda}t_0, \lambda t_1, \lambda t_2, \lambda t_3)$.
\end{itemize}
\end{cor}

Altogether, we proved:

\begin{thm}
The manifold $M_{Q}^{5}$ is homeomorphic to the quotient space 
\[
(D^2\times T^3)/\approx ,\;\;  \text{where}\;\;
  ({\bf z}, {\bf t})\approx ({\bf z}^{'}, {\bf t}^{'})\;\;  \text{iff}\;\; \]

\begin{enumerate}
 \item ${\bf z}={\bf z}^{'}\not\in S^1\; \text{and} \; {\bf t}^{'} =  {\bf t}$, 
\item ${\bf z}, {\bf z}^{'} \in S^1$  and  there exists $\lambda \in S^1$ such that  $ ({\bf {z}}^{'}, {\bf t}^{'}) = \rho _{14} (\lambda)\cdot ({\bf z}, {\bf t})$ for the representation $\rho _{14} : S^1\to T^4$ given by $\rho (\lambda) = ( \frac{1}{\lambda}, \lambda , \lambda, \lambda)$.
\end{enumerate} 
\end{thm} 



\begin{rem}
Note that by   Corollary~\ref{tildeMproduct1} the homeomorphism between the quotient  $((M^2\setminus S^1)\times T^3)/\approx$  and  $(M^2\setminus S^1)\times T^3$ does not extend to a homeomorphism between $\tilde{M}_{Q}^{5}$ and $\tilde{M}^{2}\times T^3$, that is,  to a homeomorphism between $M_{Q}^{5}$
and  $\C P^1 \times T^3$.
\end{rem}

\subsection{$M^3$ towards  $\C P^1$ and $M_{Q}^{5}$}
Recall that $M^{3}$ from Lemma~\ref{intersection} is a subspace in $\C P^{5}$ consisting of the points $(z_0 : z_1 : z_2 : |z_3| : |z_4| : |z_5|)$, where $z_k = |z_k|e^{i\psi _k}$, $k=0,1,2$, then  $|z_k|$ are given by formulas~\eqref{2345} for $k=2,3,4,5$ , and  $z_0, z_1, z_2$ satisfy equation~\eqref{surface}.   
 The conditions $z_i=0$, $i=0,1,2$ give the three circles:
\[
S^{1}_{0} = (0: \frac{e^{i\psi }}{\sqrt{6}}:  \frac{e^{i\psi }}{\sqrt{6}}: \frac{\sqrt{5}}{\sqrt{18}}:  \frac{\sqrt{5}}{\sqrt{18}}: \frac{1}{3}), \; 
S^{1}_{1}= (\frac{e^{i\psi }}{\sqrt{6}} : 0 :   \frac{e^{i\psi }}{\sqrt{6}}:  \frac{\sqrt{5}}{\sqrt{18}}: \frac{1}{3}: \frac{\sqrt{5}}{\sqrt{18}}), 
\]
\[
S^{1}_{2} = (\frac{e^{i\psi }}{\sqrt{6}}:  \frac{e^{i\psi }}{\sqrt{6}}: 0: \frac{1}{3}: \frac{\sqrt{5}}{\sqrt{18}}:  \frac{\sqrt{5}}{\sqrt{18}}).
\]

Now, the circle $S^1$ acts freely  on $M^3$ by the representation $S^1\to T^6$ given by $t\to (t,t, t, 1, 1,1)$ that is by
\begin{equation}\label{S1action}
t\cdot (z_0:z_1:z_2: |z_3|:|z_4|: |z_5|) = (tz_0 : tz_1 : tz_2 : |z_3| : |z_4|: |z_5|).
\end{equation}

\begin{prop}\label{M3S1}
The orbit space of $M^3$ by the free $S^1$-action  is $\C P^1$. 
\end{prop}

\begin{proof}
Let us consider the map $p : M^3 \to \C P^1$ given by $p(z_0:z_1:z_2:|z_3|:|z_4|:|z_5|) = (z_1|z_4|: z_0|z_5|) $. It is continuous surjection and 
\[p^{-1}(z_1|z_4| : z_0|z_5|) = \{(z_{0}^{'}:z_{1}^{'}:z_{2}^{'}:|z_{3}^{'}|:|z_{4}^{'}|:|z_{5}^{'}|)\}, \] such that $z_{0}^{'}|z_{5}^{'}| = \lambda z_0|z_5|$ and $z_{1}^{'}|z_{4}^{'}|=\lambda z_1|z_4|$ for $\lambda \in S^1$. It implies  that \[\frac{1}{9}(|z_i|^{2}-|z_{i}^{'}|^2) + (|z_i|^4-|z_{i}^{'}|^4)=0\] for $i=0,1$, which gives $|z_i|=|z_{i}^{'}|$ for $i=0,1$. Therefore, $|z_i|=|z_{i}^{'}|$, $i=2,3,4,5$ as well, and $z_{i}^{'} =\lambda z_i$ for $i=0,1$.  Since $(z_0, z_1, z_2), (z_{0}^{'}, z_{1}^{'}, z_{2}^{'})$  satisfy~\eqref{surface},  it follows $z_{2}^{'} = \lambda  z_{2}$.  Therefore, $p$ induces homeomorphism $M^3/S^1 \to \C P^1$. 
\end{proof}

We push up this further.  
\begin{prop}\label{M3S5}
The inclusion   $M^{3}\subset S^5$ induces the embedding $\C P^1\to \C P^2$ of their  orbit spaces by the  free $S^1$-action, which is given by
\[
(c : c^{'})\stackrel{\varphi}{\to} (c :c^{'}: c-c^{'}).
\]
\end{prop}
\begin{proof}
The projection $q : S^5\subset \C P^5\to \C P^2$ of the Hopf fibration $S^5 \stackrel{S^1}{\to}\C P^2$  is   given by
\[
q(z_0:z_1:z_2: |z_3|: |z_4| : |z_5|) = (z_1|z_4|: z_0|z_5|: z_2|z_3|).
\]
Namely,  $q^{-1} (z_1|z_4|: z_0|z_5|: z_2|z_3|) =\{(z_{0}^{'}:z_{1}^{'}:z_{2}^{'}:|z_{3}^{'}|:|z_{4}^{'}|:|z_{5}^{'}|)\}$ such that  $z_{0}^{'}|z_{5}^{'}|= \lambda z_0|z_5|$,  $z_{1}^{'}|z_{4}^{'}|=\lambda z_1|z_4|$ and  $z_{2}^{'}|z_{3}^{'}| = \lambda z_2|z_3|$ for $\lambda \in S^1$.  In an analogous way as in the previous proof we deduce that $z_{i}^{'}=\lambda z_i$, $i=0,1,2$. 

Let $i : M^3\to S^5\subset \C P^5$ be the inclusion. We obtain a commutative diagram
\[ \begin{tikzcd}
M^3 \arrow{r}{i} \arrow[swap]{d}{p} & S^5 \arrow{d}{q} \\%
\C P^1 \arrow{r}{\varphi}& \C P^2
\end{tikzcd},
\]
where the embedding $\varphi : \C P^1\to \C P^2$ is given as stated, because   $z_0|z_5|+z_2|z_3|=z_1|z_4|$.
\end{proof}

Since the  inclusion  $S^3\subset S^5$  induces the embedding $\C P^1\to \C P^2$ of their orbit spaces by the free $S^1$-action, and this embedding  is given by
$(c:c^{'})\to (c: c^{'}:0)$,  we deduce:

\begin{thm}
The space $M^3$ is homeomorphic to $S^3$.
\end{thm}
\begin{proof}
For the  space $M^3\subset S^5$ we consider   the map $ g: M^3\to S^3$ defined  by 
\[
g(z_0, z_1, z_2) = (\frac{z_1a(z_1)}{|z_0|^2a(z_0)^2 +|z_1|^2a(z_1)^2},  \frac{z_0a(z_0)}{|z_0|^2a(z_0)^2 +|z_1|^2a(z_1)^2}, 0),
\]
where $a(z_i) = \sqrt{\frac{1}{9}+|z_i|^2}$, $i=0,1,2$.
This map is  $S^1$-equivariant and  by  Proposition~\ref{M3S5}  induces commutative diagram 
\[ \begin{tikzcd}
M^3 \arrow{r}{g} \arrow[swap]{d}{p} & S^3 \arrow{d}{q} \\%
\C P^1 \arrow{r}{\tilde{g}}& \C P^1
\end{tikzcd},
\]
where $\tilde{g}(c:c^{'}:c-c^{'}) = (c:c^{'}:0)$ is a homeomorphism. Since the $S^1$-action is free it follows that $g$ is a homeomorphism.

\end{proof}

From Proposition~\ref{M3MQ5}  we see that the representation  $T^2\to T^6$ defined by $(t_1, t_2) \to (t_1, t_2, 1, 1, \frac{1}{t_2}, \frac{1}{t_1})$  produces free $T^2$ action on $M_{Q}^{5}$.  By~\eqref{S1action}  we have defined  free circle action on $M_{Q}^{5}$. Altogether we obtain  $T^3$-action on $M_{Q}^{5}$.  It is straightforward to see:

\begin{lem}
The   $T^3$-action on $M_{Q}^{5}$ given  by the representation $T^3\to T^6$, which is defined by $(t_1, t_2, t_3) \to (t_1t_3, t_2t_3, t_3, 1, \frac{1}{t_2}, \frac{1}{t_1})$ is free.
\end{lem} 

\begin{cor}
The orbit space of $M_{Q}^{5}$ by the given $T^3$-action is homeomorphic to $\C P^1$.
\end{cor}

\begin{proof}
This orbit space is homeomorphic to the orbit  space  $M^{3}\times T^2$  by the diagonal action of $S^1\times T^2$, and this orbit space is $M^{3}/S^1\cong \C P^1$.
\end{proof} 

Altogether we proved:

\begin{thm}\label{mainnew}
The manifold $M_{Q}^{5}$ is homeomorphic to  $S^3\times T^2$.
\end{thm}

\section{Structure of $M_{Q}^{5}$ as a  complete intersection}

The manifold $M_{Q}^{5}\subset \C ^{4} = \R^{8}$ is an affine manifold  and we denote by  $a_k=u_k+iv_k$, $1\leq k\leq 4$ the coordinates in the chart $M_{23}$.

Then  $a_1a_4-a_2a_3 = A + iB$, where $A= u_1u_4-v_1v_4-u_2u_3+v_2v_3$ and $B=u_1v_4+v_1u_4-u_2v_3-u_3v_2$.

By~\eqref{jednacine}  the manifold  $M_{Q}^{5}$ is given  by the system of equations:
\[
3= u_3^2 + v_3^2 + u_1^2 + u_2^2+4 (A^2+B^2), \]
\[
 3(u_4^2 +v_4^2)= u_3^2 + v_3^2 + 4 (u_1^2 +v_1^2) + A^2+B^2,
\]
\[
3(u_2^2 +v_2^2) = 4(u_3^2 +v_3^2)+ u_1^2 +v_1^2+ A^2+B^2.
\]
 Rewriting this, we see that $M_{Q}^{5}$ is a common level set of the functions: 
\[
f_{1}(u, v) = u_1^2+v_1^2 + u_2^2+v_2^2 -u_3^3-v_3^2-u_4^2-v_4^2,
\]
\[
f_{2}(u, v) = 5(u_1^2+v_1^2) +u_3^2+v_3^2 -4(u_4^2+v_4^2),
\]
\[
f_{3}(u, v) = 4(u_1^2+v_1^2)+u_3^2+v_3^2-3(u_4^2+v_4^3)+A^2+B^2 ,
\]
at the point $(0, -1, 0)$, where $u=(u_1, u_2, u_3, u_4)$, $v=(v_1, v_2, v_3, v_4)$, that is
\[
M_{Q}^{5} = f_{1}^{-1}(0)\cap f_{2}^{-1}(-1)\cap f_{3}^{-1}(0).
\]
 Note that $u_2^2+v_2^2 \neq 0$, $u_4^2+v_4^2\neq 0$ and $A^2+B^2+u_i^2+v_i^2\neq 0, i=1,3$,  $u_1^2+v_1^2+u_3^2+v_3^2\neq 0$.

Let $J=J(f_1, f_2, f_3)$ be the Jacobian  matrix for  the functions $f_1, f_2, f_3$. By Proposition~\ref{M3MQ5} the manifold  $M_{Q}^{5}$ can be obtained by the action of the torus $T^2$ on the surface $M^{3}$ given by~\eqref{2345}.

\begin{lem}
The rank  of $J$ at a point $t\cdot {\bf z}\in M_{Q}^{5}$, where $t\in T^2$ and ${\bf z}\in M^{3}$,  is equal to its  rank at the point ${\bf z}$.
\end{lem}

The circle $S^1$ acts freely on   $M^3$ and in the local coordinates  $a_1,\ldots ,a_4$ on $M^3$ this action is given by
\[
\lambda \cdot {\bf z} =
   (-\lambda a_3: \lambda a_1: \lambda (a_1a_2-a_3a_4) :  1 : a_4: - a_2).\]

\begin{cor}
The rank of $J$ at a point $\lambda \cdot {\bf z}\in M^{3}$,  where  $\lambda \in S^1$ and ${\bf z} \in M^{3}$, is equal to its rank at the point ${\bf z}$.
\end{cor}

 The coordinates  of the points of  the surface $M^{3}\subset M_{Q}^{5}\subset \R^8$ satisfy $v_2=v_4 = 0$ while $u_4, u_2\neq 0$.  This implies $v_1u_4=u_2v_3$,   $A=u_1u_4-u_2u_3$, $B=v_1u_4-u_2v_3$ and $u_1^2+v_1^2+u_3^2+v_3^2 +u_2^2 =\frac{1}{3}$.  In addition, 
\[
 u_1^2+v_1^2 + u_2^2 -u_3^3-v_3^2-u_4^2=0, \; \; 5(u_1^2+v_1^2) +u_3^2+v_3^2 -4u_4^2=-1,\]
\[ 4(u_1^2+v_1^2)+u_3^2+v_3^2-3u_4^2+A^2 +B^2=0.
\]

\begin{prop}
The rank of the Jacobian matrix  $J$ at any point ${\bf z}\in M^{2}$ is equal to   $3$.
\end{prop}

\begin{proof}
The Jacobian matrix $J$ at a point from $M^{2}$ is
\begin{equation}
\left(\begin{array}{cccccccc}
u_1 & v_1 & u_2 & 0 & -u_3 & -v_3 & -u_4 & 0\\
5u_1 & 5v_1 & 0 & 0 & u_3 & v_3 & -4u_4 & 0\\
4u_1 +u_4A& 4v_1 +u_4B& u_2-u_3A-v_3B & v_3A & -u_2A & -u_2B & -3u_4+u_1A+v_1B & -v_1A
\end{array}\right) \ .
\end{equation}
The minor $M_{356}$ is
\[
M_{356} =
\left | \begin{array}{ccc}
u_2 & -u_3 & -v_3\\
0 & u_3  & v_3\\
u_2-u_3A-v_3B & -u_2A & -u_2B
\end{array}\right | = u_2^{2}(v_3A -u_3B).
\]
Therefore, at  a point for which $v_3A -u_3B\neq 0$ we deduce that $\rk J =3$.

For $v_3A=u_3B\neq 0$  we consider the minor $M_{457}$, which is
\[
M_{457} =
\left | \begin{array}{ccc}
0 &  -u_3 & -u_4\\
0 & u_3 & -4u_4\\
v_3A & -u_2A & -3u_4+u_1A+v_1B
\end{array}\right | =  5u_3u_4v_3A = 5u_{3}^{2}u_4B\neq 0.
\]

Let $v_3A=u_3B = 0$. If $A=B=0$ then $u_3^2+v_{3}^2\neq 0$, say $v_3\neq 0$ we consider the minor

 \[
M_{367} =
\left | \begin{array}{ccc}
u_2 & -v_3 & -u_4\\
0 & v_3 & -4u_4\\
u_2 & 0 & -3u_4
\end{array}\right | = -u_2v_3u_4
\left | \begin{array}{ccc}
1 & -1 & 1\\
0 & 1 & 4\\
1 & 0 & 3
\end{array}\right |  = 2u_2v_3u_4.
\]

 For $A\neq 0$ and $v_3=0$ we have that $v_1=0$ as well, which implies $u_1^2+u_3^2\neq 0$.  By the action of $\lambda =e^{i\phi}\in S^1$ we obtain the point $(u_1^{'}, v_1^{'}, u_{2}, u_{3}^{'},v_{3}^{'},  u_{4})$, where  $u_{k}^{'} = u_k\cos \phi$, $v_{k}^{'} = u_k \sin \phi$ and $A^{'} = u_1^{'}u_4-u_{2}u_{3}^{'}= A\cos \phi $. The minor $M_{457}$ at this point is equal to $5u_{3}^{'}v_{3}^{'}u_{4}A^{'} =  5 u_{3}^{2}u_{4}A \sin \phi \cos^{2} \phi$.
If $\lambda = \pm 1, \pm  i$   and $u_3\neq 0$ we deduce that $M_{457}\neq 0$ so the  rank of $J$ at a point ${\bf z}\in M^{3}$ for which $A\neq 0$, $u_3\neq 0$ and $v_3=0$ is equal to $3$.   If $u_3=0$ then $u_1\neq 0$ and we consider the minor $M_{135}$ which  is 
 \[
M_{135} =
\left | \begin{array}{ccc}
u_1 & u_2 &  0\\
5u_1 & 0 & 0 \\
4u_1 & u_2 & -u_2A
\end{array}\right | = u_1u_2^2A\left | \begin{array}{ccc}
1 & 1 &  0\\
5 & 0 & 0 \\
4 & 1 & -1
\end{array}\right |  =5u_1u_2^2A.
\]
For $A=0$ and $u_3=0$ we have $u_1=0$.   By the action of $\lambda =e^{i\phi}\in S^1$ we obtain the point $(u_1^{'}, v_1^{'}, u_{2}, u_{3}^{'},v_{3}^{'},  u_{4})$, where  $u_{k}^{'} = -v_k\sin \phi$, $v_{k}^{'} = v_k \cos \phi$ and $B^{'} = v_1^{'}u_4-u_{2}v_{3}^{'}= B\cos \phi $.   The minor $M_{457}$ at this points is $5(u_{3}^{'})^2u_4B^{'} = 5v_{3}^2 u_4 B \sin ^{2}\phi \cos \phi$. For  $\lambda = \pm 1, \pm  i$   and $v_3\neq 0$ we deduce that $M_{457}\neq 0$. For $v_3=0$ we have $v_1\neq 0$ and 

 \[
M_{236} =
\left | \begin{array}{ccc}
v_1 & u_2 &  0\\
5v_1 & 0 & 0 \\
4v_1 & u_2 & -u_2B
\end{array}\right | = v_1u_2^2B\left | \begin{array}{ccc}
1 & 1 &  0\\
5 & 0 & 0 \\
4 & 1 & -1
\end{array}\right |  =5v_1u_2^2B.
\]
\end{proof}

\begin{thm}
The manifold $M_{Q}^{5}$ is a complete intersection in $\R ^{8}$.
\end{thm}

\section{Symplectic reduction on  $G_{n,2}$}
\subsection{Basic facts}
Let $(M, \omega)$ be a symplectic manifold and the compact torus $T$ acts on $M$ preserving the symplectic form $\omega$. For any $v\in \bf{t}$, the Lie algebra for $T$, by $X_{v}$ is denoted the corresponding $T$-invariant vector field on $M$. The torus action is said to be Hamiltonian if the one form $\omega (X_{v}, \cdot)$ is exact, that is, there exists a function $H_{v}$, a Hamiltonian, such that $\omega (X_{v}, Y) = dH_{v}(Y)$ for any vector field $Y$ on $M$.

For a fixed basis $\{e_i\}$ for $\bf{t}$ and the corresponding Hamiltonians $H_{e_i}$,  it is defined  the moment map
\begin{equation}\label{sympl}
\mu : M \to {\bf t}^{\ast}, \;\; \mu (x) =\mu ^{\ast}_{x}, \; \; \mu ^{\ast}_{x}(e_i) =  H_{e_i}(x).
\end{equation}
By the  theorem of Aliyah~\cite{A} and Guillemin-Sternberg~\cite{GS} the image $\mu ( M)$ is a convex polytope in ${\bf t}^{\ast}$ for a compact $M$.

Assume that the moment map $\mu$ is proper, let ${\bf x}$ be a regular value of $\mu$ and assume that $T$ acts freely on $\mu ^{-1}({\bf x})$. Then the  quotient manifold $\mu ^{-1}({\bf x})/T$ is endowed with a unique symplectic form $\hat{\omega}$, see~\cite{MW},~\cite{Me},  such that 
\[
p^{\ast} \hat{\omega} = i^{\ast}\omega ,
\]
where $i : \mu ^{-1}({\bf x}) \to M$ is the inclusion and $p : \mu ^{-1}({\bf x}) \to \mu ^{-1}({\bf x})/T$ is the projection.

This new symplectic manifold  $(\mu ^{-1}({\bf x})/T, \hat{\omega})$ is referred to as the symplectic reduction or as the symplectic quotient of $M$ by $T$-action.  Note that this symplectic quotient in general depends on a regular value ${\bf x}$.

\begin{rem} We observed in Remark~\ref{regularvalues} that on the Grassmann manifolds $G_{n,2}$ the condition that ${\bf x}\in \Delta _{n,2}$ is a regular value for $\mu$ and that the torus $T^{n-1}$ acts freely on $\mu ^{-1}({\bf x})$ are equivalent.
\end{rem}

\subsection{The case of  $G_{n,2}$}
  
A Grassmann manifold $G_{n,2} = U(n)/U(2)\times U(n-2)$, among the others,   admits a  K\" ahler metric that is invariant for the canonical $U(n)$-action and, consequently, the  $U(n)$-invariant   symplectic form. The Pl\"ucker embedding $G_{n,2}\to \C P^{N}$, $N=\binom{n}{2}-1$ is isometric with respect to such metric on $G_{n,2}$ and Fubini-Study metric on $\C P^{N}$.   We consider the effective  action of the torus $T^{n-1}=T^n/\text{diag}(T^n)$ on $G_{n,2}$, which is induced from  the coordinate wise action of $T^n\subset U(n)$ on $\C ^{n}$.   The corresponding moment map defined by~\eqref{sympl} is  $\mu : G_{n,2}\to \Delta _{n,2}$ given by
\[
\mu (L) =\frac{1}{\sum\limits_{I \in \binom{[n]}{2}}|P^{I}(L)|^2}\sum _{I \in \binom{[n]}{2}}|P^{I}(L)|^2\Lambda _I,
\]
where $\Lambda _{I}\in \Z ^n$, $\Lambda _{I}(j) =1$ for $j\in I$ and $\Lambda _{I}(j)=0$ for $j\notin I$.

We already recalled that the moment map defines the  stratification of $G_{n,2}$ and the family of admissible polytopes. This family also gives the chamber decomposition of $\Delta _{n,2}$. In the chamber notation $C_{\omega}$ below,  the index $\omega$ denotes a family of subsets $\sigma$ of $[n]$ for which  $C_{\omega}$ lies in the relative interior of an admissible polytope $P_{\sigma}$.   The following holds~\cite{GMP} - Section 4,~\cite{BT2} - Section 3,~\cite{BT2nk} - Section 7: 

\begin{thm}
Let   $C_{\omega}\subset \Delta _{n,2}$  be a chamber such that $\dim C_{\omega}=n-1$.
\begin{itemize}
\item  For any  ${\bf x}\in C_{\omega}$, the the level set $\mu ^{-1}({\bf x})\subset G_{n, 2}$ is a smooth manifold and  the manifolds $\mu ^{-1}({\bf x})$ and $\mu ^{-1}({\bf y})$ are diffeomorphic for any    ${\bf x}, {\bf y}\in C_{\omega}$.
\item For any  ${\bf x}\in C_{\omega}$ the torus $T^{n-1}$ acts freely on $\mu ^{-1}({\bf x})$ and the orbit spaces  $\mu ^{-1}({\bf x})/T^{n-1}$, $\mu ^{-1}({\bf y})/T^{n-1}$ are diffeomorphic for any    ${\bf x}, {\bf y}\in C_{\omega}$.
\end{itemize}
\end{thm}

In this way, to any $C_{\omega}$, $\dim C_{\omega}=n-1$ it can be assigned the smooth manifold $F_{\omega}$, which we refer to in~\cite{BT3} as the  space of parameters of a chamber. Following~\cite{H}, it is   pointed in~\cite{BTW}:

\begin{thm}
For any chamber $C_{\omega}$, $\dim C_{\omega}=n-1$, the manifold $F_{\omega}$ is diffeomorphic to:
\begin{itemize}
\item the geometric invariant theory quotient $\mathcal{G}(\mathcal{L})$ for the linearized diagonal action of  $PGL(2,\C)$ on $(\C P^{1})^{n}$, where the linearization $\mathcal {L}$ is given by an arbitrary point ${\bf x}\in C_{\omega}$.
\item the moduli space of weighted pointed stable genus zero curves $\overline{\mathcal{M}}_{(\mathcal{A}, 0)}$ for a weight vector $\mathcal{A}\in U(\mathcal{L})$ for a neighborhood $U$ of  ${\bf x}\in C_{\omega}$, which determines $\mathcal{L}$.
\end{itemize}
\end{thm}

From the recalled  results on symplectic reduction we deduce: 

\begin{thm}
For any chamber $C_{\omega}\subset \Delta _{n,2}$ such that $\dim C_{\omega}=n-1$ the manifold $F_{\omega}$ is a  symplectic  reduction of $G_{n,2}$ by the canonical $T^n$-action.
\end{thm}

Thus, a lot moduli spaces  $\overline{\mathcal{M}}_{(\mathcal{A}, 0)}$ for which the weight vectors $\mathcal{A}$ are very close to a maximal chamber in $\Delta _{n,2}$ are symplectic quotients of the Grassmannians $G_{n,2}$ by the canonical $T^n$-action.

\subsection{Universal space of parameters and symplectic reduction}

A manifold $F_{\omega}$ is, following~\cite{BT3}, the union of the spaces of parameters $F_{\sigma} = W_{\sigma}/(\C ^{\ast})^{n}$ of the strata $W_{\sigma}$ such that the admissible polytopes $P_{\sigma} = \overline{\mu (W_{\sigma})}$ form the chamber $C_{\omega}$, that is $C_{\omega}\subset \stackrel{\circ}{P}_{\sigma}$ for $\sigma \in \omega$ and $C_{\omega}\cap P_{\sigma}=\emptyset$ for $\sigma \not\in \omega$.

 According to~\cite{BT2} any $F_{\sigma}$  can be embedded in some   $(\C P^{1}_{B})^{l}$, where
$B=\{(1:0), (0:1)\}$, $\C P^{1}_{B} = \C P^1\setminus B$, $1\leq l\leq n-3$.    The space  of parameters of the main stratum if $F_n\cong  (\C P^{1}_{A})^{n-3}$, where $A=B\cup \{(1:1)\}$,   while the largest  spaces of parameters of the strata which are different from the main   are  $F_{\sigma} \cong  (\C P^{1}_{A})^{n-4}$  and  their admissible   polytopes are given by the  halfspaces  $x_i+x_j<1$, $1\leq i<j\leq n$.

In seeking for the largest space of parameters of the chambers, in a sense that  there exist projections from the largest one to the spaces of parameters of other chambers,   we point:

\begin{lem}
A point ${\bf x} = (x,\ldots ,x)\in \Delta _{n,2}$ belongs to a chamber of maximal dimension $n-1$ if and only if $n$ is odd  and $x=\frac{2}{n}$ . This  chamber is  given by  the intersection of half-spaces $x_{i_1}+\ldots 
+x_{i_l}<1$ for $2\leq l\leq [\frac{n}{2}]$. 
\end{lem}
\begin{proof}
Since $nx=2$ it follows ${\bf x} = (\frac{2}{n}, \ldots, \frac{2}{n})$. For any $2\leq k\leq n-2$ the sum $kx = \frac{2k}{n}$ might be $1$ only when $n$ is even. In addition, for even $n=2k$ we have $kx=1$, so ${\bf x}$ does not belong to a chamber of maximal dimension. For odd $n$, we consider the chamber $C_{\omega}$ given by  the intersection of half-spaces $x_{i_1}+\ldots +x_{i_l}<1$ for $2\leq l\leq [\frac{n}{2}]$, following   Remark~\ref{arrangement}.  Then, ${\bf x}\in C_{\omega}$.
\end{proof}

\begin{cor}\label{odd}
For odd $n$, the space of parameters of the chamber $C_{\omega} :  x_{i_1}+\ldots +x_{i_l}<1$ for $2\leq l\leq [\frac{n}{2}]$ of dimension $n-1$ is the largest space of parameters of a chamber.
\end{cor}

\begin{proof}  The admissible polytopes in $\Delta _{n,2}$ can be, according to~\cite{BT3},   described by the hyperplane arrangement from Remark~\ref{arrangement}.  More precisely, any admissible polytopes of maximal dimension $n-1$ can be obtained as the intersection of admissible family of half spaces $H _{I_1}: \langle v_{I_1}, x\rangle < 1, \ldots , H_{I_l}: \langle v_{I_l},  x\rangle <1$, determined by    hyperplanes from the arrangement given by Remark~\ref{arrangement}. A chamber $\C_{\omega}^{'}$ of maximal dimension is by Remark~\ref{arrangement}  chamber of the given hyperplane arrangement. 
Let for a chamber $\C_{\omega}^{'}$ we have $x_{i_1}+\ldots +x_{i_l}>1$ for some $2\leq l\leq [\frac{n}{2}]$. Then $x_{j_1}+\ldots +x_{j_{n-l}}<1$, were $\{j_1, \ldots , j_{n-l}\} =\{1, \ldots , n\}\setminus \{i_1, \ldots , i_l\}$. The space of parameters of the stratum whose   admissible polytope  is given by this half space is, according to~\cite{BT3},  homeomorphic  to $(\C P^{1}_{A})^{l-1}$. On the other hand the space of parameters of the stratum  whose  admissible polytope is given  by  $x_{i_1}+\ldots +x_{i_l}<1$ is homeomorphic to $(\C P^{1}_{A})^{n-l-1}$.   Since $ l\leq [\frac{n}{2}]$,  the statement follows. In the same way we argue for an admissible polytope given by the intersection of half spaces.
\end{proof}

In an analogous way we prove:

\begin{cor}\label{even}
For even $n$,  the space of parameters of the chamber $C_{\omega _0} :  x_{i_1}+\ldots +x_{i_l}<1$ for $2\leq l\leq \frac{n}{2} -1$ of dimension $n-1$  is the largest space of parameters of a chamber.
\end{cor}

The following is obvious:

\begin{rem}\label{Comegaaction}
The chamber  $C_{\omega_0}$ is invariant under the action of the symmetric group $S_n$ on $\R^n$ which permutes the coordinates.
\end{rem}

Let 
\[
\tilde{F}_{\omega_0} = \{ F_{\sigma}, \; \sigma\in \omega_{0}, \sigma \neq \binom{[n]}{2}\}.
\]
From Remark~\ref{Comegaaction} it follows:

\begin{cor}\label{Somega}
The  $S_n$-action on $C_{\omega_0}$ induces $S_n$-action on the set $\tilde{F}_{\omega_0}$.
\end{cor}

In~\cite{BT2nk} we introduced  the notion of the universal space of parameters $\mathcal{F}$,  which is a compactification of the space of parameters of the main stratum. In the case of $G_{n,2}$ we proved~\cite{BT3} that the universal space of parameters $\mathcal{F}_{n}$ is a smooth manifold coinciding  with the Deligne-Mumford compactification of the moduli space of distinct ordered $n$ points in $\C P^1$, which is given by the moduli space  $\overline{\mathcal{M}}_{n,0}$ of  stable pointed genus zero curves. The space $\overline{\mathcal{M}}_{n,0}$ can be described as  $\overline{\mathcal{M}}_{\mathcal{A},0}$, where $\mathcal{A} = (1, \ldots, 1)$.

In addition, in~\cite{BTW} we proved that $\mathcal{F}_{n}$ can be obtained by the wonderful compactification~\cite{LL} of  $\bar{F}_{n}\subset (\C P^{1})^{N}$, $N=\binom{n}{2}-1$. Here $\bar{F}_{n}$ is the closure in $\C P^{N}$  of the space of parameters $F_n$  of the main stratum. More precisely, the main stratum  $W_{n}\subset G_{n,2}$ is given by $W_{n} = \{L\in G_{n,2} \; | \: P^{I}(L)\neq 0, I\in \binom{[n]}{2}\}$ and \[F_{n} =W_{n}/(\C ^{\ast})^n = \{(c_{ij}:c_{ij}^{'})_{3\leq i<j\leq n}\in (\C P^1)^{N}, c_{ij}c_{ik}^{'}c_{jk} = c_{ij}^{'}c_{ik}c_{jk}^{'}\},\]
where $c_{ij}, c_{ij}^{'}\neq 0, c_{ij}\neq c_{ij}^{'}.$

In~\cite{BT2} we proved that for any chamber $C_{\omega} =  \cap \stackrel{\circ}{P}_{\sigma} \subset \Delta _{n,2}$ it holds that $\mathcal{F}_{n} = \cup _{\sigma \in \omega} \tilde{F}_{\sigma}$ and this union is disjoint. Here  $\tilde{F}_{\sigma}$ are virtual spaces of parameters and there exist projections $p_{\sigma} : \tilde{F}_{\sigma}\to F_{\sigma}$, which define projection $\mathcal{F}_{n}\to F_{\omega}$. In particular the quotients of $\tilde{F}_{\sigma}$ defined by the projections $p_{\sigma}$ give the outgrows in the compactification $F_{\omega}$ to $\mathcal{F}_{n}$. Denote by $\mathcal{O}_{\omega}$ the set of outgrows in the compactification $F_{\omega}$ to $\mathcal{F}_{n}$ that is
\[
\mathcal{O}_{\omega} = \{ \tilde{F}_{\sigma}/p_{\sigma}, \; \sigma \in \omega\}.
\]

Recall~\cite{BT1},~\cite{BT2} that for the space of parameters $F_n$ of the main stratum it holds $\tilde{F}_{n} = F_{n}$. Then from Corollary~\ref{Somega} we deduce:

\begin{cor}
The set of outgrows $\mathcal{O}_{\omega_0}$  for the chamber $C_{\omega_0}$    is invariant under the action of  $S_n$.
\end{cor}

\begin{cor}
The manifold $F_{\omega_0}$ can be obtained from $\mathcal{F}_{n}$ by factorizing  the set $\mathcal{O}_{\omega}$ by the  $S_n$-action  followed  by  the projections $p_{\sigma } : \tilde{F}_{\sigma}\to F_{\sigma}$, where  $\sigma =\{\{1,2,3\}, \{1,2, 3, 4\}, \ldots , \{1, 2,\ldots ,[\frac{n}{2}]-1\}\}$.
\end{cor}

Further, from Corollary~\ref{odd} and Corollary~\ref{even} we have:
\begin{prop}
The  universal space of parameters $\mathcal{F}_{n} =  \overline{\mathcal{M}}_{n,0}$ is isomorphic to the  space of parameters of a chamber in $\Delta _{n,2}$ if and only if $n=4, 5$.  
\end{prop}
\begin{proof}
We proved in~\cite{BT3} that for any space  $F_{\omega}$ there exists continuous surjective map $p_{\omega} : \mathcal{F}_{n}\to F_{\omega}$,  which projects the  virtual spaces of parameters $\tilde{F}_{\sigma}$ to the spaces of parameters $F_{\sigma}$ for $\sigma \in \omega$.   Therefore,  we consider the largest space of parameters which corresponds to the chamber given by  Corollary~\ref{odd} or Corollary~\ref{even}. For $n\geq 6$, an admissible polytope $P_{\sigma}$ given by $x_i+x_j+x_k<1$ contributes to this chamber, its space of parameters  $F_{\sigma}\cong F_{n-4}\subset  (\C P^{1}_{A})^{n-4}$, while its virtual space  of parameters is \[\tilde{F}_{\sigma}\cong \{(c_{ij}:c_{ij}^{'})_{3\leq i<j\leq n}\in  \C P^1\times (\C P^{1}_{A})^{n-4}, c_{ij}c_{ik}^{'}c_{jk} = c_{ij}^{'}c_{ik}c_{jk}^{'}\}.\] It implies that the projection $\tilde{F}_{\sigma}\to F_{\sigma}$ is not one-to-one, that is $p_{\omega}$ is not a
homeomorphism.
\end{proof}

\begin{cor}\label{univsimpl}
A  manifold $\mathcal{F}_{n} =  \overline{\mathcal{M}}_{0,n}$   is a  symplectic quotient of $G_{n,2}$ by the canonical $T^{n}$-action if and only if $n=4,5$.
\end{cor}

The Losev-Manin spaces $\bar{L}_{0,n}$ can be interpreted~\cite{LM},~\cite{M},~\cite{BTW}  as the moduli space of weighted pointed stable curves with a weight vector $\mathcal{A} = (1,1, a_3, \ldots, a_n)$, $\sum _{i=3}^{n}a_i <1$ and the space $\overline{\mathcal{M}}_{0,n}$   can be obtained by the methods of  wonderful compactification applied to  $\bar{L}_{n,2}$.  We showed in~\cite{BTW}  that $\bar{L}_{0,n}$ is isomorphic to  the space of parameters of a chamber for $G_{n,2}$ if and only if $n=5$.

\begin{cor}
A Losev-Manin space $\bar{L}_{0,n}$  is a symplectic quotient of $G_{n,2}$ by the canonical $T^n$-action if and  only if  $n=5$.
\end{cor}

In the case $G_{4,2}$ the symplectic quotient by the canonical $T^4$-action does not depend on a regular value of the moment map. For $n =  5$ this is not the case any more.

\begin{itemize}
\item The universal space of parameters $\mathcal{F}_{5} =\overline{\mathcal{M}}_{0, 5}$ is the del Pezzo surface of degree $5$ and it can be obtained by blowing up the surface \[\bar{F}_{5} = \{(c_{1}, c_{1}^{'}), (c_{2}:c_{2}^{'}), (c_{3}:c_{3}^{'})\in (\C P^{1})^3 \; |\; c_1c_{2}^{'}c_3 = c_{1}^{'}c_{2}c_{3}^{'}\}\] at the point $((1:1), (1:1), (1:1))$.  By Corollary~\ref{univsimpl}, the manifold $\mathcal{F}_{5}$ is a symplectic reduction.  More precisely, it is a symplectic reduction of a point ${\bf x}$ from the chamber  $C_{\omega _0}$ which is given as the intersection of halfspaces $x_i+x_j<1$, $1\leq i<j\leq 5$, In particular, $\mathcal{F}_{n}= \mu ^{-1}(\frac{2}{5}, \ldots, \frac{2}{5})/T^5$. 
\item  On the other hand,~\cite{BTW},   the manifold $\bar{L}_{0, 5} = \bar{F}_{5}$ is the space of parameters of the chamber $C_{\omega _1}\subset \Delta _{5,2}$ given by the intersection of halfspaces $x_1+x_2>1$, $x_i+x_+j<1$  for $\{i, j\}\neq \{1,2\}$, so $\bar{F}_{5}$ is as well symplectic reduction of $G_{5,2}$ by the canonical $T^5$-action.  
\item The manifolds $\mathcal{F}_{5}$ and $\bar{L}_{0,5}$ are different, more precisely, by~\cite{BT2} it holds $\mathcal{F}_{n} = \text{Bl}_{\text{point}}\bar{L}_{0,5}$

\end{itemize}

\bibliographystyle{amsplain}

 Victor M.~Buchstaber\\
Steklov Mathematical Institute, Russian Academy of Sciences\\ 
Gubkina Street 8, 119991 Moscow, Russia\\
E-mail: buchstab@mi.ras.ru
\\ \\ 

Svjetlana Terzi\'c \\
Faculty of Science and Mathematics, University of Montenegro\\
Dzordza Vasingtona bb, 81000 Podgorica, Montenegro\\
E-mail: sterzic@ucg.ac.me

\end{document}